\documentclass[10pt]{article}
 \usepackage{indentfirst}
\title{AlMOST SURE ONE-ENDEDNSS OF A RANDOM GRAPH MODEL OF DISTRIBUTED LEDGERS }
\author{J. Feng, C. King and K. R. Duffy\\
\\
Department of Mathematics \\
Northeastern University \\
MA 02115, USA}
\date{}

\usepackage{amsmath,amssymb,amsthm}
\usepackage{textcomp} %straight quot by \textquotesingle

\usepackage{graphicx}
\usepackage{tabularx} 
\usepackage{setspace} 
\usepackage{tikz} %for drawing graph with nodes
\definecolor{alizarin}{rgb}{0.82, 0.1, 0.26}
\definecolor{ao(english)}{rgb}{0.0, 0.5, 0.0}
\usetikzlibrary{graphs}
\usepackage{float}
\usepackage{amsfonts}

\newcommand{\ba}{\begin{align}} 
\newcommand{\ea}{\end{align}}

\def\be{\begin{eqnarray}}
\def\ee{\end{eqnarray}}
\def\bee{\begin{eqnarray*}}
\def\eee{\end{eqnarray*}}

\def\P{\mathbb{P}}

\newtheorem{thm}{Theorem}[section]
\newtheorem{cor}[thm]{Corollary}
\newtheorem{lem}[thm]{Lemma}

\newtheorem{prop}[thm]{Proposition}
\newtheorem{defn}[thm]{Definition}

\begin{document}%\recd{}{}%Do not alter this line.

\maketitle

\begin{abstract}
Blockchain and other decentralized databases, known as distributed ledgers, are designed to store information online where all trusted network members can update the data with transparency. The dynamics of ledger's development can be mathematically represented by a directed acyclic graph (DAG). One essential property of a properly functioning shared ledger is that all network members holding a copy of the ledger agree on a sequence of information added to the ledger, which is referred to as consensus and is known to be related to a structural property of DAG called one-endedness. In this paper, we consider a model of distributed ledger with sequential stochastic arrivals that mimic attachment rules from the IOTA cryptocurrency. We first prove that the number of leaves in the random DAG is bounded by a constant infinitely often through the identification of a suitable martingale, and then prove that a sequence of specific events happens infinitely often. Combining those results we establish that, as time goes to infinity, the IOTA DAG is almost surely one-ended.\\
\textit{Keywrods}: blockchain, IOTA, stochastic directed acyclic graph, martingale 
\end{abstract}

\section{Background} % Initial capital letter, then lower case. No full stop.

% Write the text of your paper using normal LaTeX commands.
% For instance, you can use the `\cite' command~\cite{ref1}.
% When giving citations a numbering system is preferred~\cite{ref2},
% but an author--date system is also acceptable~\cite{ref3}.

A distributed ledger is a decentralized database where transactions are stored on a directed acyclic graph (DAG). The goal of any distributed ledger is to provide a secure and consistent record of transactions. Due to the widespread adoption of the methods for cryptocurrencies, there has been growing interest recently in formally establishing properties of the ledger \cite{Fer19,Go20,SS22,SS23,Zh20}.

In the DAG associated to a distributed ledger, each vertex represents a block or package of information. Each new vertex represents a new transaction, and is attached to one or more existing blocks according to a random attachment rule. The attachment mechanism also guarantees that a transaction will be finally linked to the selected existing blocks only after this new transaction finishes a time consuming task called proof of work (POW). It is the delay time that results from the POW that complicates the dynamics of distributed ledger. 

To better understand POW, consider the following description: first, a new vertex $A$ arrives and chooses one or more existing blocks in DAG to be attached to, and we call the selected blocks the parents of vertex $A$. Secondly, using information in $A$ and its parents, a question is generated. The user trying to upload this vertex $A$ will start solving this time consuming problem. Only after the question is solved (i.e. POW is completed), directed edges from $A$ to its parents are created, indicating that the POW for $A$ with its selected parents is finished, and therefore vertex $A$ is accepted into the DAG. The solution to the POW problem will also be stored in $A$ so that any changes in the data of $A$ or its parents will yield a different question such that the stored solution will no longer be correct and therefore all members of the ledger will know the altered data is invalid. This mechanism helps users to verify the data in the ledger and protect its record from being doctored. In order to doctor a transaction in a vertex $B$, the actor would need to solve all of the POW problems again for that block and any following blocks that is connected to vertex $B$, which would require a tremendous amount of computational power. Therefore, as more and more vertices establish path(s) toward a block, this block becomes increasingly reliable and resistant to manipulation. When considering distributed ledger, if a vertex is linked by any future vertex, the transactions in it are considered verified since at least one POW is finished to secure the data. Figure \ref{soD} provides an example of DAG.

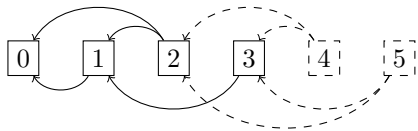
\begin{figure}[H]
\centering
\begin{tikzpicture}

    \node (0) [rectangle,draw] at (0,-4) {0};
    \node (1) [rectangle,draw] at (1,-4) {1};
    \node (2) [rectangle,draw] at (2,-4) {2};
    \node (3) [rectangle,draw] at (3,-4) {3};
    \node (4) [rectangle,draw,dashed] at (4,-4) {4};
    \node (5) [rectangle,draw,dashed] at (5,-4) {5};

    \graph{
        (1)->[bend left=60](0),
        (2)->[bend right=60](1),
        (2)->[bend right=60](0),
        (3)->[bend left=60](1),
        (4)->[bend right=60,dashed](2),
        (4)->[bend right=60,dashed](3),
        (5)->[bend left=60,dashed](2),
        (5)->[bend left=60,dashed](3)
    
    };
\end{tikzpicture}
\caption{A solid directed edge implies that the assciated POW has been completed and the data has been accepted to the ledger. For example, vertex 3 has selected 1 as its parent and finished its POW. A dashed vertex with outgoing dashed edge implies the POW has not yet been finished. For example, vertex 5 has selected 2 and 3 as its parents but its POW has not yet been finished.}
\label{soD}
\end{figure}

Due to the wide use of distributed ledger technologies, rigorously establishing mathematical properties related to their security and stability are of increasing interest. A vertex, corresponding to a ledger entry, becomes more reliable and resistant to manipulation as more vertices are connected to it through directed edges. In a mathematical model where vertices are added sequentially, questions of consistency and reliability can be framed in terms of the resulting DAG. If a vertex is connected by all but finitely many future vertices through directed edges, we call it as confirmed vertex, which speaks to the reliability of the related ledger entry. POW is a time consuming process which requires a great amount of computational power. When a vertex v finishes its POW, the number of previous vertices that are connected to v represents the amount of data secured by this POW with its consumed computational power. For a confirmed vertex, nearly all the computational power consumed in the future will be used to secure this confirmed vertex, hence the proportion of the confirmed vertices can be used to represent the efficiency in using the computational power. A structural property of an infinite DAG called one-endedness is also an indicator for the security of ledger. While a precise definition will be provided in Section \ref{sectionmain}, a heuristic description of one-endedness is that any two infinite paths in the DAG can both have an infinite overlap with a third path. In contrast to one-endedness, a graph has multiple ends when there exist two infinite paths such that any path intersect at least one of them only finitely many times. Illustrative examples of a DAG with one-ended property and one without are shown in Figure \ref{fig6}, if a graph is not one-ended, there will be infinitely many non-confirmed vertices which indicates the lack of efficiency in using the computational power to secure the ledger. In this paper, both confirmed vertices and one-ended property are analyzed.

\begin{figure}[H]
\centering
\begin{tikzpicture}

    \node (00) [rectangle,draw] at (0,0) {0};
    \node (01) [rectangle,draw] at (1,0.5) {1};
    \node (02) [rectangle,draw] at (1,-0.5) {2};
    \node (03) [rectangle,draw] at (2,0.5) {3};
    \node (04) [rectangle,draw] at (2,-0.5) {4};
    \node (05) [rectangle,draw] at (3,0.5) {5};
    \node (06) [rectangle,draw] at (3,-0.5) {6};
    \node (07) [rectangle,draw] at (4,0.5) {7};
    \node (08) [rectangle,draw] at (4,-0.5) {8};   

    \node (09) [rectangle] at (5,0.5) {...};
    \node (091) [rectangle] at (5,-0.5) {...};  
    
    \node (10) [rectangle,draw] at (7,0) {0};
    \node (11) [rectangle,draw] at (8,0.5) {1};
    \node (12) [rectangle,draw] at (8,-0.5) {2};
    \node (13) [rectangle,draw] at (9,0.5) {3};
    \node (14) [rectangle,draw] at (9,-0.5) {4};
    \node (15) [rectangle,draw] at (10,0.5) {5};
    \node (16) [rectangle,draw] at (10,-0.5) {6};
    \node (17) [rectangle,draw] at (11,0.5) {7};
    \node (18) [rectangle,draw] at (11,-0.5) {8};
    \node (19) [rectangle] at (12,0.5) {...};
    \node (191) [rectangle] at (12,-0.5) {...}; 
    
    \graph{
        (01)->(00),(02)->{(00),(01)},(03)->{(01),(02)},(04)->(02),
       (05)->(03),(06)->{(04),(05)},(07)->{(05),(06)},(08)->(06),
        (11)->(10),(12)->(10),(13)->(11),(14)->(12),(15)->(13),(16)->(14),(17)->(15),(18)->(16)
    };
\end{tikzpicture}
\caption{Here, directed edges describe connections and paths follow the reverse direction of directed edges. The graph on the left is one-ended where the paths $(0,1,3,5,...)$ and $(0,2,4,6,...)$ intersect with the path (0,1,2,3,5,6,7,...) infinitely many times, and hence they are considered paths within the same end. The graph on the right provides an example of a DAG without one-ended property as there is no path that intersects both $(0,1,3,5,...)$ and $(0,2,4,6,...)$ infinitely many times and hence these two paths are in two different ends. Furthermore, all vertices in the right graph are not confirmed vertices.}
\label{fig6}
\end{figure}
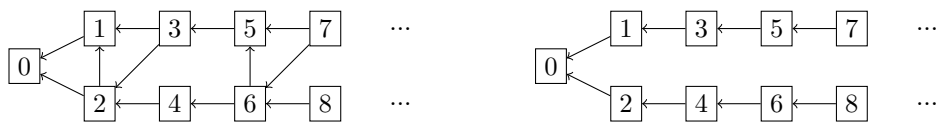

The ledger, in the form of the DAG and vertex contents, are typically stored in a peer to peer network, where a peer refers to the user that holds a local copy of the ledger and peers communicate to each other from time to time to update their local copy. At any one time, peers however, hold different copies of the ledger because of the frequency of communication. Therefore it is important to examine whether all peers agree on infinitely many vertices as time goes to infinity, which is refer to as the consensus property \cite{Go20}. Gopalan, Sankararaman, Elwalid and Vishwanath \cite{Go20} introduced mathematical definitions related to consensus and identified the importance of one-ended property of the DAG associated with the ledger under the assumption that the duration of POW is zero. 
 
The structure of the DAG develops differently depending on the parent selection algorithm as it determines all the edges in the graph. The Bitcoin system uses the algorithm called Nakamoto construction \cite{Naka08} where only one parent is selected within the vertices that have paths with maximum distance to the vertex 0 and different rules can be used to break ties. Another construction called the \textit{throughput optimal policy} \cite{Go20, LY15} assumes that all leaves in the graph, i.e. all the vertices that have not yet been attached by future vertices, are selected as parents for the arriving vertex. It is shown by \cite{Go20} that under a peer-to-peer setting, where multiple users hold different subsets of the DAG, if vertices arrive stochastically and choose parents based on the graph copy from one of the peers, the resulting infinite graph constructed by either the Nakamoto construction or the throughput optimal policy is almost surely one-ended.

While the one-endedness property has been established for those two attachment rules, there is another popular system called IOTA cryptocurrency or the tangle \cite{Popov16}. In it, each new vertex selects two parents at random with replacement from the set of leaves of the DAG. This algorithm introduces additional variability of creation of the edges and therefore complicates the DAG model. Considering two distinct models of attachment, Partha and Aditya \cite{PS20} discuss the one-endedness of the DAG modeling the tangle. The first model, called the backward model, assumes that an arriving vertex gets attached to its parents immediately while the parents are selected based on the state of the DAG at a earlier time. For example the parents of vertex arriving at $t$ are selected based on the DAG at $t-\epsilon$ and the corresponding edges are created at $t$. This model is designed to encapsulate the delay time during communication between peers to update the state of the ledger, and hence when the vertex arrives, only a previous version of the DAG instead of the lasted version are observed. It is shown that the infinite DAG with backward model is almost surely one-ended. The second model is called the forward model, it assumes that arrivals select parents at the current time, but the edges are created after a random time. For example, a vertex arrives at time $t$, but edges between the vertex and its parents are created at $t+\theta$. This model considers the duration of POW and how it effects the development of the DAG. For the forward model, it is suggested in \cite{Popov16} that with the proof used in backward model, one-endedness can be established if one assumes there is a positive possibility that an arrival can finish POW before the next arrival occurs, i.e. the POW might be finished almost instantly.

While \cite{Popov16} suggested a root to prove one-endedness for the forward model under the assumption that there is a positive probability that the duration of POW is shorter than the vertex inter-arrival time, it is natural to assume otherwise because the POW is designed to be time consuming for the purpose of security and vertices are usually arrives at a high rate such that there are multiple arrivals within the duration of POW. Therefore it is important to examine whether one-endedness holds without this assumption on the POW process. 

In this paper, we provide a proof of one-endedness for an infinite DAG modeling IOTA cryptocurrency assuming both the peer-to-peer setting and POW duration without the need to make the assumption on minimal POW time, where the definitions of the model will be provided in Section \ref{model}. The main idea is that by analyzing the evolution equations of the tangle, we show that the number of leaves is bounded above by a constant infinitely often. Whenever the number of leaves in below this bound, any event that is measurable within a constant time will happen with probability bounded by a positive constant from below and thus occurs infinitely often. The event we are interested in is mainly requires that all the arrivals within a fixed time interval choose their parents so that any future vertices will have a path to all the vertices in the past. This event builds up a \textit{bottleneck} structure of the graph that almost surely ensures one-endedness. Each interval related to bottleneck event acts like web connecting different sub-DAGs that combines all the branches into one frequently that one-endedness is ensured. Using the same construction we are also able to analyze the existence of the confirmed vertices and their population. The main result of this paper can be found in Theorems \ref{mainresult} and \ref{confirmed}.

\section{Definition of the model} \label{model}

We now introduce the dynamics of the DAG generated by a stochastic model of the IOTA system with deterministic inter-arrival time, multinomial POW duration, and random frequency of reconciliation of the ledger within peer-to-peer network. Let $t_n:=n$ for $n=0,1,2,...$. At time $t_0=0$, the DAG is assumed to consist of a single vertex labeled $0$. We assume a vertex labeled by $i$ arrives at each time $t_i$ for $i=1,2,3,...$ and starts its POW, where vertices are only added to the DAG at the time corresponding to the completion of their POW. With $\mathcal{V}(t_0):=\{0\}$, we define $\mathcal{V}(t_n)$ to be the set of vertices at $t_n$ whose POW are completed.  Let $(i,j)$ represent a directed edge from vertex $i$ to $j$, where  vertex $j$ is a parent of $i$. With $\mathcal{E}(t_0):=\emptyset$, $\mathcal{E}(t_n)$ is defined to be the set of the directed edges which we will refer to as \textit{solid edges}. We use these solid edges to represent finished POW relation between vertices. We will focus on the discrete stochastic process  $\mathcal{G}(t_n):=(\mathcal{V}(t_n),\mathcal{E}(t_n))$ that describe the evolution of DAG.

Each new arriving vertex $i$ independently chooses its duration of POW from a finite set of integer times. Let $(\Theta_i)_{i\in\mathbb{N}}$ be i.i.d multinomial random variables taking values in $\{h_1,h_2,...,h_M\}$ where $\Theta_i$ represents the duration required for the POW of vertex $i$, where $0<h_1<h_2<...<h_M$, $h_j\in\mathbb{N}$ and $\P(\Theta_i=h_j):=p_{\Theta,j}>0$. Without loss of generality we will assume $M\geq 2$ since the case with $M=1$ assume that the POW duration is fixed. We call an arrival with POW duration $h_i$ a Type $i$ arrival. We define
$N_i(t_n) := 1_{\{\Theta_n=h_i\}}$ where $1_{\{\}}$ is the indicator function. By construction, $\sum_{i=1}^M N_i(t_n) = 1$ as only one of the $N_i$ will be 1. On arrival, each vertex $i$ randomly chooses two parents with replacement from a subset of the graph which will be described later in equation (\ref{parentselection}). Although we will focus on the parent selection algorithm used in IOTA where only two potential parents are selected, the results and proofs can be generalized to the case considering any finite number of parents, which will be stated in Prop \ref{generization}. We use the random variables $X_i$ and $Y_i$ taking values in $\mathcal{V}(t_i)$ to denote the two parents for vertex $i$, where $X_i$ and $Y_i$ can take the same value which results in a single edge coming out from vertex $i$. 

After the POW of vertex $i$ is completed, the vertex is connected to the DAG with the edges from it to its parents, resulting in:
\be
\mathcal{V}(t_{n+1})&=&\mathcal{V}(t_{n})\cup\{k|t_k+\Theta_k=t_n\}\label{evov}\\
\mathcal{E}(t_{n+1})&=&\mathcal{E}(t_n)\cup\{(k,X_k),(k,Y_k)|t_k+\Theta_k=t_n\}.\label{evoe}
\ee

In addition to the essential notations, we will need additional variables to analyze the stochastic process describing the temporal development of the DAG. For $n=0,1,...$, let $\mathcal{V}'(t_n):=\{i|i<n\}\setminus\mathcal{V}(t_n)$ denote the set of vertices at $t_n$ whose corresponding POW have not been completed, and $\mathcal{E}'(t_n):=\{(i,X_i),(i,Y_i)|i\in \mathcal{V}'(t_n)\}$ denote set of directed edges that start from vertices in $\mathcal{V}'(t_n)$ and we call edges in $\mathcal{E}'(t_n)$ as \textit{dashed edges}. We define $\mathcal{G}'(t_n):=(\mathcal{V}'(t_n),\mathcal{E}'(t_n))$ which, together with $\mathcal{G}(t_n)$, can be used to produce a picture as in Figure \ref{soD}. At each time, the arriving vertex is added into $\mathcal{V}'$ and any vertex with completed POW is removed from $\mathcal{V}'$ and placed in $\mathcal{V}$. Similarly, directed edges are added into $\mathcal{E}'$ when the corresponding POW starts and they are removed when the POW are finished. These mechanisms are described by:
\bee 
\mathcal{V}'(t_{n+1})&=& \{i\in \mathcal{V}'(t_n) | t_i+\Theta_i>t_n\}\cup \{n\}\\
\mathcal{E}'(t_{n+1})&=& \{(i,j)\in \mathcal{E}'(t_{n})| t_i+\Theta_i>t_n\}\cup \{(n,X_n),(n,Y_n)\}.
\eee 

In the DAG $\mathcal{G}(t_n)=(\mathcal{V}(t_n),\mathcal{E}(t_n))$, let $\mathcal{L}(t_n):=\{v\in{\mathcal{V}(t_n)|(i,v)\notin \mathcal{E}(t_n) \forall i}\}$ represent the set of vertices with in-degree 0, i.e. the vertices that has no directed edge pointed toward them, which we call \textit{tips}. We use $L(t_n):=|\mathcal{L}(t_n)|$ to record the number of tips, and among the tips we distinguish \textit{pending} and \textit{free tips}. Define $\mathcal{W}(t_n) := \{i\in \mathcal{L}(t_n)|  (j,i)\in \mathcal{E}'(t_n) \text{ for some } j\}$ to represent the set of pending tips and $W(t_n):=|\mathcal{W}(t_n)|$ for number of pending tips. A tip is pending at time $t_n$ if it has been selected for POW by any transaction that arrived at some time $t_j$ with $j < n$. Define $\mathcal{F}(t_n):=\mathcal{L}(t_n)\setminus\mathcal{W}(t_n)$ to be the set of free tips and $F(t_n):=|\mathcal{F}(t_n)|$ to be the number of free tips. A tip is free at time $t_n$ if it has not been selected for POW before time $t_n$. It follows that $L(t_n)=F(t_n)+W(t_n)$. 

In this paper we consider a peer-to-peer network setting where infinite number of peers each stores a part or all of the ledger and each of them communicates to others from time to time to update its local realization of the ledger by taking the union of all the graphs from all peers. Let $(\epsilon_i)_{i\in \mathbb{N}}$ be i.i.d multinomial random variables that take values in a finite set $I_\epsilon\subset \mathbb{N}$ where $\min(I_\epsilon):=\epsilon_{min}>0$, $\max(I_\epsilon):=\epsilon_{max}$ and $\P(\epsilon_i=j):=p_{\epsilon,j}>0$ for $j\in I_\epsilon$. Using the convention that $G(t_n):=G(t_0)$ for $n<0$, we assume the parents of vertex $i$ are selected based on $G(t_i-\epsilon_i)$ to model the following situation: parents are chosen based on the DAG stored in one of the many peers which was updated via communication at $t_i-\epsilon_i$. To model the IOTA construction, the two parents for vertex $i$  are selected with replacement among the set of tips at time $t_i-\epsilon_i$ , $\mathcal{L}(t_i-\epsilon_i)$, with equal probability, that is 
\be
\P(X_i=j|\mathcal{G}(t_i-\epsilon_i))=\P(Y_i=j|\mathcal{G}(t_i-\epsilon_i)):=\frac{1}{L(t_i-\epsilon_i)}1_{\{j\in{\mathcal{L}(t_i-\epsilon_i)}\}}\label{parentselection}
\ee

The random variables $\Theta_n, \epsilon_n, X_n,Y_n$ determine the evolution of the graph. Let $\Sigma_i$ denote the $\sigma$-algebra generated by all the random variables $\Theta_n,\epsilon_n,X_n$ and $Y_n$ for $n=1,2,...,i$. Using the distributions of $\Theta_n, \epsilon_n,X_n,Y_n$ introduced, we have that for $k\in\{1,2,...,M\}$ and $j\in I_\epsilon$,
\be
\P(\Theta_n=h_k, \epsilon_n=j, X_n=v,Y_n=v'|\Sigma_{n-1})=\frac{p_{\Theta,k}\times p_{\epsilon,j}}{L(t_n-j)^2}1_{\{v,v'\in{\mathcal{L}(t_n-j)}\}}.\label{evolutiondistribution}
\ee
Among the tips being selected at each step, we define 
\be 
\delta_n:=|\{X_n,Y_n\}\cap\mathcal{F}(t_n)|\label{definedelta}
\ee
to record the number of free tips at time $t_n$ that are selected as parents by the vertex $n$.

%Among the tips being selected at each step, we use $F_i(t_n)$ to indicate the number of free tips selected among $\mathcal{F}(t_n)$ by this the arriving vertex $n$, where the subscript $i$ indicates the Type of the arrival. To describe this, define 
%\bee F_i(t_n):=|\{j|x_n=j \text{ or }y_n=j \text{ for } j\in \mathcal{F}(t_n)\}|\cdot 1_{\{\theta_n=h_i\}} \eee 
%For example, if one free tip and one pending tip are selected as parents by an Type 2 arrival at $t_n$ then we say $F_2(t_n)=1$ and $F_i(t_n)=0$ for $i\neq 2$. Note that if the arrival select a free tip as its only parent, i.e. $x_n=y_n\in\mathcal{F}(t_n)$, then we still have that the number of free tips selected is 1. 

We now describe how the function $F(t_n)$ is updated at each time step. Consider the sequence $(F(t_n))_{n\in\mathbb{N}}$, which increases because some POWs are completed and the corresponding vertices are added to $\mathcal{G}(t_{n+1})$. On the other hand, the number of free tips decreases because some free tips are selected as parent for the arriving vertex and hence they become pending tips, which is calculated by equation (\ref{definedelta}). So we have
\be
F(t_{n+1}) - F(t_n) = \left(\sum_{i=1}^{M}N_i(t_n - h_i)\right)  - \delta_n.\label{evolution}
\ee
The first summation identifies the arrivals that have just finished their proof of work at time $t_n$. The second term $\delta_n$ identifies the free tips being selected as parents at time $t_n$. It is worth mentioning that the free tips selected as parents will become pending tips at the next step, and so they will be counted as part of $W(t_{n+1})$, but we need not write out the update equation for $W$ since the one for $F$ is enough for the analysis performed in this paper.

For the stochastic process $(\mathcal{G}(t_n))_{n\in\mathbb{N}}$, we define an appropriate limit as time goes to infinity. Let $\mathcal{S}_*$ denote the space of connected DAGs rooted at vertex $0$ with all vertices having finite degrees. Define $d_*(\mathcal{G}_1,\mathcal{G}_2):=(r+1)^{-1}$, where $r$ is the biggest integer such that the two r-balls rooted at vertex $0$ in $\mathcal{G}_1$ and $\mathcal{G}_2$ are identical.  As established in \cite{Da07}, the metric space $(\mathcal{S}_*,d_*)$ is separable and complete, i.e. a Polish space. By construction of our model, $\mathcal{G}(t_i)\subseteq \mathcal{G}(t_j)$ for $i<j$, and hence, together with Lemma \ref{cauchy1}, it can be shown that the sequence $(\mathcal{G}(t_i))_{i\in\mathbb{N}}$ is almost surely Cauchy and hence the limit exist and equals to $\cup_{i=1}^{\infty} \mathcal{G}(t_i)$ as stated in Lemma \ref{temp5}. More details of the proof of the existence of limiting DAG is included in Theorem \ref{mainresult}.

Here we provide some terminologies related to directed graph for convenience of discussion: Let $\mathcal{V}$ and $\mathcal{E}$ denote the set of vertices and the set of directed edges in a DAG $\mathcal{G}$, and a pair $(v_i,v_j)$ denotes a directed edge from $v_i$ to $v_j$. A \textit{directed path} from $v_n$ to $v_1$ is a sequence $(v_n,v_{n-1},...,v_1)$ such that $(v_{i+1},v_{i})\in \mathcal{E}$ for all $i=1,2,3,...,n-1$. A vertex $v_1$ is said be to $\textit{reachable}$ from $v_n$ if there exist a directed path from $v_n$ to $v_1$.

\section{Main result}\label{sectionmain}
We first recall some definitions related to one-endedness including rays, an equivalent relation for rays and one-ended property.\cite{Hr64}
\begin{defn}\label{d31}
A \textit{ray} in a directed graph is an infinite sequence of vertices $v_0,v_1,v_2,...$ in which either $(v_{i+1},v_i)$ is a directed edge of the graph for all $i=0,1,...$, or $(v_{i},v_{i+1})$ is a directed edge of the graph for all $i=0,1,...$
\end{defn}

\begin{defn}\label{d32}
Two rays $r_1$ and $r_2$ in a DAG are equivalent if there is a ray $r_3$ (not necessarily distinct from $r_1$ and $r_2$) s.t. $r_3$ intersect $r_1$ and $r_2$ infinitely often.
\end{defn}

\begin{defn}\label{d33}
The DAG is one-ended if all rays are equivalent.
\end{defn}
Armed with these definitions, we are in a position to state our main result, Theorem \ref{mainresult}, on the one-endedness of the stochastic DAG describing an IOTA-like ledger will be provided as follows.
\begin{thm}\label{mainresult}
Under the assumptions of the model defined in section \ref{model}, for any value of $M$ which records the number of choices of POW time, letting the tangle grow as $t\rightarrow\infty$, the corresponding infinite DAG  $\mathcal{G}(\infty):=\lim_{n\rightarrow\infty}\mathcal{G}(t_n)$ exist almost surely in the metric space $(\mathcal{S}_*,d_*)$, $\mathcal{G}(\infty)=\cup_{i=1}^{\infty} \mathcal{G}(t_i)$ and it is one-ended almost surely.
\end{thm}

In order to prove Theorem \ref{mainresult}, we introduce some lemmas and construct a sequence of events that joins all rays together and happens infinitely often. Lemma \ref{randomwalk} recounts a suppermartingale result and it is used to prove Lemma \ref{L<b} which states that the number of tips is bounded above by a finite constant infinitely often. This will allow us to identify a sequence of events that join all rays together, prove that these events happen infinitely often as stated in Lemma \ref{Binfinitelyoften}, and then finally prove one-endedness using a critical property described in Lemma \ref{temp4}. At the same time, analyzing this sequence of events yields Lemma \ref{cauchy1} which further proves existence of the limiting DAG, $\mathcal{G}(\infty)$.
\begin{lem}\label{randomwalk}
Let $b<b_0\in \mathbb{R}$ and $Y(t_n)$ be a supermartingale such that $Y(t_0)=b_0$, $Y(t_{n+1})-Y(t_n)$ takes values in $\{-1,0,1\}$ and $P(Y(t_{n+1})-Y(t_n)=-1|\{Y(t_k):k \leq n\})> p'$ for some $p'>0$. Defining $\tau=\inf\{n|Y(n)\leq b\}$, we have $P(\tau<\infty)=1$. 
\end{lem}

Lemma \ref{randomwalk} is a standard result of supermartingale with bounded difference, which can be proved using Doob's optional stopping time lemma \cite{Williams91}. We will relate the stochastic process $L(t_n)$, which records the number of tips, to the proccess $F(t_n)$, which records the number of free tips and can be related to a supermartingale satisfying the conditions in Lemma \ref{randomwalk}. This idea will help us prove the following lemma:
\begin{lem} \label{L<b}
If $b\in(10 h_M-6h_1+3M\epsilon_{max} +2,\infty)$ and define event $\mathcal{A}_i=\{L(t_i)\leq b\}$. Then the sequence of events $\{\mathcal{A}_i\}_{i=1}^\infty$ happens infinitely often almost surely.
\end{lem}
\begin{proof}
By the definition, $L(t_n)=F(t_n)+W(t_n)$ where $W(t_n)$ records the number of pending tips. First we show that $W(t_n)\leq 2h_M$, and then we relates $F(t_n)$ to a supermartingale in order to show that $F(t_n)$ is bounded above by a constant infinitely often, finally the results for $W(t_n)$ and $F(t_n)$ proves Lemma \ref{L<b}.

First we show that $W(t_n)\leq 2h_M$. Note that the longest time for a pending tip to remain as a tip is $h_M$, hence an vertex $v\in\mathcal{W}(t_n)$ must satisfy that $\min\{t_k|v\in\mathcal{W}(t_k)\}\in [t_n-h_M+1,t_{n}]$, i.e. a pending tip at time $t_n$ implies that it was first selected as parents and became a pending tip at some time within $[t_n-h_M+1,t_n]$. Hence using equation (\ref{definedelta}) we conclude that $W(t_n)\leq\sum_{k=t_n-h_M}^{t_n-1}\delta_n$. Since at most two free tips can be selected as parents at each step, we have $\delta_n\leq2$. Therefore, $W(t_n)\leq\sum_{k=t_n-h_M}^{t_n-1}\delta_{k}\leq 2h_M$.

With the fact that $W(t_n)\leq 2h_M$, in order to prove Lemma \ref{L<b}, it suffices to show $F(t)\leq a$ for some $a\in(8 h_M-6h_1+3M\epsilon_{max} +2,\infty)$ infinitely often. The idea is to compare the process $F(t)$ to a supermartingale $Y(t_n)$ which satisfies the conditions in Lemma $\ref{randomwalk}$ such that $|F(t_n)-Y(t_n)|$ is bounded by a constant. By the fact that $F(t)$ is close to the constructed supermartingale which reaches a low value within finite time with probability 1, the result follows. 

Suppose that at $t_\alpha$, $F(t_\alpha)>a$. We wish to show that if $a>8 h_M-6h_1+3M \epsilon_{max} +2$, then $F(t_n)\leq a$ for some finite $n>\alpha$ with probability 1. If this is true for any $\alpha\in(h_M,\infty)$, then for any $k\in \mathbb{N}$ we have $P(\cap_{n=k}^{\infty}\{F(t_n)>a\})=0$, which implies $P(\cup_{k=1}^{\infty}\cap_{n=k}^{\infty}\{F(t_n)>a\})=0$ and hence $P(\{F(t_n)\leq a\} \text{ infinitely often})=1$ which would prove Lemma \ref{L<b}.

Recall that $N_i(t_j)=1$ if the vertex $j$ has that the duration of POW $\Theta_j$ is $h_i$ and $N_i(t_j)= 0$ otherwise. Recall the evolution equation of $F(t_n)$ described in equation (\ref{evolution}) and that the term  $\sum_{i=1}^{M}{N_i(t_n-h_i)}$ records the number of vertices whose corresponding POWs have just been finished. We can count the value of $\sum_{i=1}^{M}{N_i(t_n-h_i)}$ with a different idea: For any vertex $j\in[n-h_M,n-h_1]$ we check if the vertex $j$ has POW duration being $n-j$, i.e. 
\be
F(t_{n+1})-F(t_{n})=-\delta_{n}+\sum_{i=1}^{M}{N_i(t_n-h_i)}=-\delta_{n}+\sum_{j=n-h_M}^{n-h_1}\left(\sum_{i:h_i=n-j}N_i(t_j)\right).\label{recount}
\ee
For any $n>\alpha$, the same idea for equation (\ref{recount}) applies and we conclude that
\be
F(t_n)-F(t_\alpha)&=&-\sum_{j=\alpha}^{n-1}\delta_j + \sum_{j=\alpha}^{n-1}\sum_{i=1}^{M}{N_i(t_j-h_i)}\nonumber \\
&=&-\sum_{j=\alpha}^{n-1}\delta_j+\sum_{j=\alpha-h_M}^{n-1-h_1}\left(\sum_{i:h_i+j\in[\alpha,n-1]}N_i(t_j)\right). \label{doublesum}
\ee
The two double summations in equation (\ref{doublesum}) calculate the number of vertices that finish their POWs during the interval $[t_\alpha,t_{n-1}]$. We separate the possible vertices that might finish their POWs during the interval $[t_\alpha,t_{n-1}]$ into two cases: Case 1, if for the arriving vertex at the time $t_j$, it is the case that $t_j+[h_1,h_M]\subseteq [t_\alpha,t_{n-1}]$, which corresponds to the case that the vertex finishes its POW within the interval $[t_\alpha,t_{n-1}]$ regardless of the duration of POW, then this vertex will contribute 1 increment in the value of $F(t_n)-F(t_{\alpha})$; Case 2, for any vertex $j$ such that $t_j+[h_1,h_M]\cap [t_\alpha,t_{n-1}]\neq \emptyset$ and $ t_j+[h_1,h_M]\not\subseteq [t_\alpha,t_{n-1}]$, which means this vertex may or may not finish POW during $[t_\alpha,t_{n-1}]$ depending on duration of POW, this vertex contributes either 0 or 1 increment in equation (\ref{doublesum}). Note that if a vertex $j$ is neither in case 1 nor case 2 then $t_j+[h_1,h_M]\cap [t_\alpha,t_{n-1}]=\emptyset$ and it will not be considered in equation (\ref{doublesum}). Therefore the number of vertices that is considered in case 1 is at least $(n-\alpha)-(h_M-h_1)$ and the number of vertices that is considered in case 2 is at most  $2(h_M-h_1)$. Hence
\be
(n-\alpha)-(h_M-h_1)\leq \sum_{j=\alpha-h_M}^{n-1-h_1}\left(\sum_{i:h_i+j\in[\alpha,n-1]}N_i(t_j)\right)\leq n-\alpha+(h_M-h_1)\label{randombound}
\ee

Equation (\ref{randombound}) says that the number of vertices whose POWs have just been finished is within a constant error from the number $(n-\alpha)$. This enables us to define the process $Y(t_n)$ that starts at time $t_\alpha$ with the same value of $F(t_\alpha)$ and evolves distinctly as described in the following equations for $t_n>t_\alpha$:
\be
Y(t_n)&=&F(t_\alpha)+\sum_{j=\alpha}^{n-1} 1-\sum_{j=\alpha}^{n-1}\delta_j,\\
Y(t_n)-Y(t_{n-1})&=&1-\delta_{n-1},\label{evomartingale}
\ee 
then by equation (\ref{randombound}), $|F(t_n)-Y(t_n)|\leq 2(h_M-h_1)=:\Delta_{Y,F}$ for $n>\alpha$.

Next we wish to show that if $Y(t_k)>2h_M+3M\epsilon_{max}+2+\Delta_{Y,F}$, then $E[Y(t_{k+1})|\{Y(t_j):t_j\leq t_{k}\}]<0$ which means $Y(t_{k+1})$ is a supermartingale. We show this using the idea that if $Y(t_k)>2h_M+3M\epsilon_{max}+2+\Delta_{Y,F}$, then the number of free tips $F(t_k)$ is sufficiently large to make sure that the term $\delta_k$ used in equation (\ref{evomartingale}) has higher probability to be 2 than to be 1.

Let $\Sigma_i$ denote the $\sigma$-algebra generated by all the random variables $\Theta_k,\epsilon_k,X_k$ and $Y_k$ for $k=1,2,...,i$. By equation (\ref{doublesum}), we have $F(t_i)\in \Sigma_{i-1}$. Since a vertex $v\in \mathcal{F}(t_k)$ is not an element of $\mathcal{F}(t_k-\epsilon_k)$ only if the POW of vertex $v$ is completed and added to the set of free tips at some time $t\in [t_k-\epsilon_k+1,t_k]$. Also, note that $\sum_{i=1}^{M}{N_i(t_j-h_i)}\leq M$, i.e. there are at most $M$ vertices finishing POW and becoming free tips at each time. Therefore, suppose that $F(t_k)> 2 h_M+ 3M \epsilon_{max}+2=:a^*$, then 
\be
|\mathcal{F}(t_k-\epsilon_{k})\cap\mathcal{F}(t_k)|>2 h_M+2M\epsilon_{max}+2.\label{bound1}
\ee  
Similarly, a vertex $v'$ is an element in $\mathcal{F}(t_{k}-\epsilon_{k})\setminus\mathcal{F}(t_k)$ only if another vertex arrives and makes $v'$ become a pending tip by selecting $v'$ as parent at some time within the interval $[t_k-\epsilon_k+1,t_k]$. Since there is one vertex arrives at each step and it has at most two parents, 
\be
|\mathcal{F}(t_{k}-\epsilon_{k})\setminus\mathcal{F}(t_k)|\leq 2\epsilon_{max}<2M\epsilon_{max}.\label{bound2}
\ee
Recall that $W(t)$ records the number of pending tips which is always bounded by $2h_M$. Using equations (\ref{bound1}) and (\ref{bound2}) we get that $F(t_k)> 2 h_M+ 3M \epsilon_{max}+2=:a^*$ implies,
\be 
|\mathcal{F}(t_k-\epsilon_{k})\cap\mathcal{F}(t_k)|>2 h_M+2M\epsilon_{max}+2>|\mathcal{F}(t_{k}-\epsilon_{k})\setminus\mathcal{F}(t_k)|+W(t_{k}-\epsilon_k).
\ee
This implies that if $F(t_k)>a^*$, then the set $\mathcal{F}(t_k-\epsilon_{k})\cap\mathcal{F}(t_k)$ includes over half of the set of tips $\mathcal{L}(t_k-\epsilon_k)$. Hence if $F(t_k)>a^*$, $\delta_k$, the number of free tips in $\mathcal{F}(t_k)$ selected as parents at time $t_k$, has higher probability to be 2 than to be 0, which is equivalent to saying that for any event $D_{k-1}\in \Sigma_{k-1}$ and $D_{k-1}\subseteq \{F(t_k)> a^*\}$,
\bee
\P(\delta_{k}=2|D_{k-1})>\frac{1}{4}>\P(\delta_{k}=0|D_{k-1}) \quad \text{  and } \quad E(\delta_j|D_{k-1})>0
\eee
Therefore, if $Y(t_k)> a^*+\Delta_{Y,F}$, then $F(t_k)>a^*$ and hence the random variable $Y(t_{k+1})$ described in equation (\ref{evomartingale}) is a supermartingale with bounded difference satisfying Lemma \ref{randomwalk} with $p'=1/4$.

As a reminder, we want to show that given a constant $a>8 h_M-6h_1+ 3M \epsilon_{max} +2$, if $F(t_\alpha)>a$ at some time $t_\alpha$ then $F(t_n\leq a)$ for some $n\in(\alpha,\infty)$ almost surely. Recall that $\Delta_{Y,F}:=2(h_M-h_1)$ and $a^*:=2 h_M+ 3M \epsilon_{max}+2$. The condition $a>8 h_M-6h_1+3M \epsilon_{max} +2=a^*+3\Delta_{Y,F}$ implies that $Y(t_\alpha)=F(t_\alpha)> a^*+3\Delta_{Y,F}$. Therefore, by Lemma \ref{randomwalk}, the process $Y(t_n)$ starting at time $t_\alpha$ with value $F(t_\alpha)>a$ will become less than or equal to $a^*+2\Delta_{Y,F}$, which is less than $a$, within finite time and $Y(t_n)> a^*+\Delta_{Y,F}$ is true before $Y(t_n)$ gets less or equal to $a^*+2\Delta_{Y,F}$ to make sure $Y(t_n)$ remains as a supermartingale. Thus $F(t_n)\leq Y(t_n)+\Delta_{Y,F}=a^*+3\Delta_{Y,F}\leq a$ within finite time after $t_\alpha$ which implies $F(t_n)\leq a$ infinitely often. This with $L(t_n)=F(t_n)+W(t_n)$ and $W(t_n)\leq 2h_M$ together conclude Lemma \ref{L<b}. 

\iffalse
Here we also provide a proof for finite hitting time of bias dependent random walk.
Let $b_{1}<0$ and $b_{2}>0$, $M(0)=0$, $E[M(n+1)|\{M(k):k \leq n\}]\leq 0$, $P(M(n+1)-M(n)=-1|\{M(k):k \leq n\})\geq \frac{1}{4}$ and $|M(n+1)-M(n)|$ takes values in $\{-1,0,1\}$. Define $\tau_{1}=\inf\{n|M(n)\leq b_{1}\}$,$\tau_{2}=\inf\{n|M(n)\geq b_{2}\}$ and  $\tau=min(\tau_{1},\tau_{2})$.\\
Since there exist $K=\tau_{1}-\tau_{2}$ and $\delta=4^{-K}>0$, such that for every n, $\P(\tau\leq n+K|\{M(k):k \leq n\})>\delta$.
 %Probability with Martingales page 101
Hence $E[\tau]<\infty$ and since the increment of such random walk is bounded by 1, Doob's optional stopping time lemma implies:
\bee 
E[M(\tau)]\leq E[M(0)]=0.
\eee 
By definition:
\bee 
E[M(\tau)]&=&b_{1}\P(\tau_{1}<\tau_{2})+b_{2}\P(\tau_{2}<\tau_{1})\\
0&\geq& b_2+(b_{1}-b_2)\P(\tau_1<\tau_2)\\
\P(\tau_1<\tau_2)&\geq &\frac{b_2}{b_2-b_1}
\eee 
Letting $b_2\rightarrow \infty$, $\P(\tau_1<\tau_2)\rightarrow \P(\tau_1<\infty)\geq 1$.
\fi

\end{proof}

Lemma \ref{L<b} shows that the number of tips is bounded above by a constant $b\in(10 h_M-6h_1+3M\epsilon_{max} +2,\infty)$ infinitely often, which we use here on. Next, for each $i\in \mathbb{Z}^*$ with $\mathcal{A}_i=\{L(t_i)\leq b\}$ occurs, we construct an event $\mathcal{B}_i$ within an interval with fixed length so that all rays are joined together, and we call these constructed events the \textit{bottleneck events}. If the sequence of events $\{\mathcal{B}_i\}_{i=1}^\infty$ happens infinitely often, it can be shown that the limiting DAG is one-ended. 

Given $i$ such that $\mathcal{A}_i=\{L(t_i)\leq b\}$ occurs, we now construct the event $\mathcal{B}_i$ in three main steps, A, B and C through time intervals $[t_i,t_i+\kappa_A)$, $[t_i+\kappa_A,t_i+\kappa_B)$ and $[t_i+\kappa_B,t_i+\kappa_C]$ respectively as in Figure \ref{interval}, where the constants $\kappa_A,\kappa_B,\kappa_C$ will be defined in their corresponding steps. Here we provide a summary of the construction of bottleneck events:
\renewcommand{\labelenumi}{\Alph{enumi})}
\begin{enumerate}
    \item Tidying phase: Conditioned on $\{L(t_i)\leq b\}$, we identify a finite $\kappa_A$ and an event $A_i$ within interval $[t_i,t_i+\kappa_A)$ such that at time $t_i+\kappa_A$ all unfinished POW has duration $h_M$ and the number of free tips $F(t_i+\kappa_A)>3(h_M+\epsilon_{min})$. It will be shown that any $\kappa_A>1+\max\{2h_M,3h_M(h_M+\epsilon_{min})/(h_M-1)\}$ is sufficient for our goal. The identified event $A_i$ is described by equations (\ref{stepA1}) and (\ref{stepA}) while properties of step A are given in Proposition \ref{conditionA} and Corollary \ref{FA>}.
    \item Preparation phase: Following the event in step A, we show that for $\kappa_B=\kappa_A+\epsilon_{min}+1$, it is possible to construct an event $B_i$ within the interval $[t_i+\kappa_A,t_i+\kappa_B)$ such that a subset with size $2(h_M+\epsilon_{min})$ of $\mathcal{F}(t_i+\kappa_A)$ remain unselected to be parents by the time $t_i+\kappa_B$. The identified event $B_i$ is desribed by equations (\ref{stepBrule1}) and (\ref{stepBrule}) while properties of step B are given in Lemmas \ref{conditionB} and \ref{temp2}.
    \item Interchange phase: Following the events in step A and B, we introduce a constant $\kappa_C$, whose value will be given in equation (\ref{kappac}), and a construction of event $C_i$ within the interval $[t_i+\kappa_B,t_i+\kappa_C]$ such that any vertex that arrive after time $t_i+\kappa_C$ has a path to any vertex in $\mathcal{V}(t_i)$. The constructed event $C_i$ is described by equations (\ref{stepcrule1}), (\ref{Xselect}) and (\ref{Yselect}). Properties of the event introduced in this part are given in Lemmas \ref{conditionC} and \ref{temp3}.
\end{enumerate}

\iffalse
We define a function $\chi: \mathbb{N}\rightarrow \mathcal{F}(t_i)$ by:
\bee 
\chi(n)=\chi_n=v \quad\quad \text{such that } |\{v'\in\mathcal{F}(t_i)|v'\leq v\}|=n,
\eee 
and we will use $\chi_n$ to denote the vertex $v\in\mathcal{F}(t_i)$ following the definition of the function $\chi$.
\fi

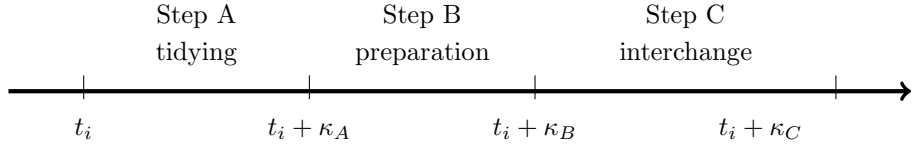
\begin{figure}[H]
\begin{tikzpicture}
    % draw horizontal line   
    \draw[ultra thick, ->] (5,0) -- (17,0);
    
    % draw vertical lines
    \foreach \x in {6,9,12,16}
    \draw (\x cm,6pt) -- (\x cm,-3pt);
    
    \node at (4,0) {\quad};
    \node at (6,-0.5) {$t_i$}; 
    \node at (9,-0.5) {$t_i+\kappa_A$}; 
    \node at (12,-0.5) {$t_i+\kappa_B$}; 
    \node at (15,-0.5) {$t_i+\kappa_C$}; 
    \node at (7.5,0.5) {tidying}; 
    \node at (10.5,0.5) {preparation}; 
    \node at (14,0.5) {interchange}; 
    \node at (7.5,1) {Step A}; 
    \node at (10.5,1) {Step B}; 
    \node at (14,1) {Step C}; 
    
\end{tikzpicture}
\caption{Step A and B are tidying and preparation phases which provide sufficient conditions of the DAG for the construction in Step C. Step C identifies a special event that join all rays together.}\label{interval}
\end{figure}

\textbf{Step A (Tidying Phase)}: Let $\kappa_A$ be a constant whose value will be identified in terms of parameters $h_M$ and $\epsilon_{min}$ in Corollary \ref{FA>}, where $h_M$ is the maximun duration of POW and $\epsilon_{min}$ is the shortest time an arriving vertex can go back for observing the DAG and selecting parents. Recall that $\mathcal{L}(t)$ denotes the set of tips and $\mathcal{F}(t)$ denotes the set of free tips. We define $A_i$ to be the event that each vertex $j\in[t_i,t_i+\kappa_A)$ follows the following conditions:
\begin{align}
\epsilon_j&=\epsilon_{min}\quad \& \quad \Theta_j=h_M>1\label{stepA1}\\
X_j&=Y_j\in\left\{
\begin{array}{cc}
 \mathcal{L}(t_j-\epsilon_{min})\setminus \mathcal{F}(t_j) &\quad \text{if }\mathcal{L}(t_j-\epsilon_{min})\nsubseteq \mathcal{F}(t_j)\\
\mathcal{L}(t_j-\epsilon_{min}) &\quad \text{if }\mathcal{L}(t_j-\epsilon_{min})\subseteq \mathcal{F}(t_j)
\end{array}
\right.\label{stepA}
\end{align}
which means each vertex $j$ selects parents within the set of tips $\mathcal{L}(t_j-\epsilon_{min})$ such that $X_j=Y_j$ and $X_j,Y_j\notin \mathcal{F}(t_j)$ whenever $\mathcal{L}(t_j-\epsilon_{min})\nsubseteq \mathcal{F}(t_j)$, otherwise vertex $j$ selects any vertex in $\mathcal{L}(t_j-\epsilon_{min})$ as its only parent. Recall that $\Sigma_i$ denotes the $\sigma$-algebra generated by all the random variables $\Theta_k,\epsilon_k,X_k$ and $Y_k$ for $k=1,2,...,i$. By equations (\ref{evolutiondistribution}), (\ref{stepA1}) and (\ref{stepA}), : 
\begin{prop}\label{conditionA}
 For any event $D_{i-1}\in \Sigma_{i-1}$ and $D_{i-1}\subseteq \{L(t_i)\leq b\}$, we have $P(A_i|D_{i-1})>0$.  
\end{prop}

The purpose of step A is so that the DAG develops up to the time $t_i+\kappa_A$ such that all vertex with unfinished POW have maximum duration of POW $h_M$ and the number of free tips $F(t_i+\kappa_A)$ exceeds a threshold $3 (h_M+\epsilon_{min})$ in order to make sure the existence of the set $\mathcal{F}_i^B$ defined by equation (\ref{unselectedset}) in step B and that the set has positive probability to remain unselected as parents by the end of step B. Lemma \ref{temp1} establishes a property of the event $A_i$ which will be used to find the value of $\kappa_A$.
\begin{lem}\label{temp1}
Suppose $\mathcal{A}_i=\{L(t_i)\leq b\}$ occurs and all vertices that arrive within $[t_i,t_i+\kappa_A)$ with $\kappa_A>2h_M+2$ satisfy equations (\ref{stepA1}) and (\ref{stepA}). For $t_j\in(t_i+h_M,t_i+\kappa_A-h_M)$, if $F(t_{j+1})-F(t_j)=0$, then $F(t_{j+k})-F(t_{j+k-1})=1$ for $k=2,...,h_M$, i.e. the number of free tips will keep increasing by 1 for the next $h_M-1$ steps. 
\end{lem}
\begin{proof}
Recall $\mathcal{V}'(t_k)$ denotes the set of vertices with unfinished POW at time $t$. Notice that if $\kappa_A>h_M$, then for any vertex $j\in \mathcal{V}'(t_k)$ with $k\in (i+h_M,i+\kappa_A]$, vertex $j$ must arrived and started its POW after time $t_i$ and hence $\Theta_j=h_M$. Hence for $t_n\in(t_i+h_M+1,t_i+\kappa_A)$, the number of vertices that complete their POWs at each step denoted by $\sum_{k=1}^M N_k(t_n-h_k)$ has only one term $N_M(t_n-h_M)$ to be 1 and therefore $\sum_{k=1}^M N_k(t_n-h_k)=1$. Together with equations (\ref{definedelta}) and (\ref{evolution}), we have that if $t_n\in(t_i+h_M+1,t_i+\kappa_A)$ and equations $(\ref{stepA1})$ and $(\ref{stepA})$ are true, then the increment of number of free tips will be either 1 or 0 depending on which condition of equation (\ref{stepA}) is satisfied, i.e.
\begin{equation*} 
F(t_{n+1})-F(t_n)=\left\{
\begin{array}{cc}
1 &\text{if } \mathcal{L}(t_n-\epsilon_{min})\nsubseteq \mathcal{F}(t_n)\\
0 &\text{if } \mathcal{L}(t_n-\epsilon_{min})\subseteq \mathcal{F}(t_n)
\end{array}\right..
\end{equation*}
Suppose $F(t_{j+1})-F(t_j)=0$ for some $t_j\in(t_i+h_M,t_i+\kappa_A-h_M)$ which means a tip $v\in \mathcal{F}(t_j)$ is selected as parent because $\mathcal{L}(t_j-\epsilon_{min})\subseteq \mathcal{F}(t_j)$, and hence $v$ becomes a pending tip at $t_{j+1}$. Since $v\in\mathcal{F}(t_j)$ and $v\in\mathcal{L}(t_j-\epsilon_{min})$, $v\in \mathcal{F}(t_k)\subseteq\mathcal{L}(t_k)$ for $k=j, j-1,...,j-\epsilon_{min}$. Since vertex $v$ is first selected as parent at $t_j$ by vertex $j$ which has POW duration $h_M$, then for $k=j+1,...,j+h_M$, $v\in\mathcal{W}(t_k)$ and hence $v\notin\mathcal{F}(t_k)$. Therefore, for $k=2,...,h_M$, $v$ will be selected as parent by the arriving vertex $j+k-1$ because of equation (\ref{stepA}) and $F(t_{j+k})-F(t_{j+k-1})=1$.
\end{proof}
Within event $A_i$ that takes place within $[t_i,t_i+\kappa_A)$, Lemma \ref{temp1} says $F(t_n)$ is a non decreasing sequence for $t_n>t_i+h_M$ which can only maintain the same value for at most one step. Hence we can make the number of free tips exceed any number by extending the duration of step A, which is described in the following corollary.
\begin{cor}\label{FA>}
Suppose $L(t_i)\leq b$ and $A_i$ defined by equations (\ref{stepA1}) and (\ref{stepA}) occur. For any constant $\kappa_A>1+\max\{2h_M,3h_M(h_M+\epsilon_{min})/(h_M-1)\}$, we have $F(t_i+\kappa_A)-F(t_i+h_M)> 3 (h_M+\epsilon_{min})$ and further $F(t_i+\kappa_A)> 3 (h_M+\epsilon_{min})$.
\end{cor}
We will assume the selectin of a $\kappa_A$ that satisfies Corollary \ref{FA>} from here on.

\textbf{Step B (Preparation Phase)}: Given step A has occurred, for $\kappa_B:=\kappa_A+\epsilon_{min}+1$, we work on the vertices that arrive within the interval $[t_i+\kappa_A,t_i+\kappa_B)$. The goal of step B is to define a set $\mathcal{F}_i^{B}$ satisfying the requirements described in equation (\ref{unselectedset}) and show that the event $B_i$ defined in equations (\ref{stepBrule1}) and (\ref{stepBrule}), where all the vertices in $\mathcal{F}_i^{B}$ remain unselected until the time $t_i+\kappa_B$, has positive probability. Since $F(t_i+\kappa_A)> 3 (h_M+\epsilon_{min})$, we can define a subset $\mathcal{F}_i^{B}\subset \mathcal{F}(t_i+\kappa_A)\subset\mathbb{Z}$ such that
\be
|\mathcal{F}_i^{B} |=2(h_M+\epsilon_{min}) \quad\quad \&\quad\quad \max \mathcal{F}_i^{B}<\min (\mathcal{F}(t_i+\kappa_A)\setminus\mathcal{F}_i^{B}).\label{unselectedset}
\ee 
 We now define $B_i$ to be the event where any vertex $j$ arrives during the interval $[t_i+\kappa_A,t_i+\kappa_B)$ satisfies the following conditions:
\begin{align} 
\Theta_j&=h_M \quad \&\quad \epsilon_j=\epsilon_{min},\label{stepBrule1}\\
X_j&=Y_j\in\{v\in\mathcal{L}(t_j-\epsilon_{min})|v\notin \mathcal{F}_i^{B}\}.\label{stepBrule}
\end{align}
Recall that $\Sigma_i$ denotes the $\sigma$-algebra generated by all the random variables $\Theta_k,\epsilon_k,X_k$ and $Y_k$ for $k=1,2,...,i$. The conditional probability of event $B_i$ is given as follows:
\begin{lem}\label{conditionB}
For any event $D_{i-1}\in \Sigma_{i-1}$ and $D_{i-1}\subseteq \{L(t_i)\leq b\}$, we have $P(B_i|D_{i-1}\cap A_i)>0$.
\end{lem}

\begin{proof}
By equations (\ref{evolutiondistribution}), (\ref{stepBrule1}) and (\ref{stepBrule}), it suffices to show that there exists a vertex $v\in\mathcal{L}(t_j-\epsilon_{min})$ such that $v\notin\mathcal{F}_i^{B}$ for all $j\in [t_i+\kappa_A,t_i+\kappa)$. For $j\in[i+\kappa_A,i+\kappa_B)=[i+\kappa_A,i+\kappa_A+\epsilon_{min}+1)$, we are interested in $\mathcal{L}(t_j-\epsilon_{min})$ where
\bee
t_j-\epsilon_{min}\in[t_i+\kappa_A-\epsilon_{min},t_i+\kappa_A].
\eee 
By Corollary \ref{FA>} and equation (\ref{unselectedset}), $|\mathcal{F}_i^{B}|=2(h_M+\epsilon_{min})<3(h_M+\epsilon_{min})<F(t_i+\kappa_A)$, hence $|\mathcal{F}(t_i+\kappa_A)\setminus\mathcal{F}_i^{B}|>h_M+\epsilon_{min}$. The case where $t_j-\epsilon_{min}=t_i+\kappa_A$ is immediate because $|\mathcal{F}(t_i+\kappa_A)\setminus\mathcal{F}_i^{B}|>0$, hence we can only consider $t_j-\epsilon_{min}\in[t_i+\kappa_A-\epsilon_{min},t_i+\kappa_A)$. For any $v'\in \mathcal{F}(t_i+\kappa_A)\setminus\mathcal{F}_i^{B}\subset\mathcal{F}(t_i+\kappa_A)$, suppose the vertex $v'$ was added to the set of free tips at some time $t'$ because the corresponding POW is finished, then $v'\in \mathcal{L}(t)$ for any $t\in [t', t_i+\kappa_A]$. Hence $v'\notin\mathcal{L}(t_j-\epsilon_{min})$ only if the POW of the vertex $v'$ was finished and added to the set of free tips at some time within $[t_j-\epsilon_{min}+1,t_i+\kappa_A]$. Together with the fact that only one POW can be finished at each step by equations (\ref{stepA1}) and (\ref{stepBrule1}), we have 
\bee 
|(\mathcal{F}(t_i+\kappa_A)\setminus\mathcal{F}_i^{B})\cap\mathcal{L}(t_j-\epsilon_{min})|\geq (h_M+\epsilon_{min})-[t_i+\kappa_A-(t_j-\epsilon_{min})]\geq h_M>0.
\eee 
Therefore we can choose $v\in(\mathcal{F}(t_i+\kappa_A)\setminus\mathcal{F}_i^{B})\cap\mathcal{L}(t_j-\epsilon_{min})$ and Lemma (\ref{conditionB}) follows.
\end{proof}

Now that we have established that the event $B_i$ has positive conditional probability, in order to introduce step C we first discuss the identification of a set of vertices whose elements have already arrived by the time $t_i+\kappa_B$. As a reminder, $\mathcal{V}'(t)$ is the set of vertices whose POWs have not been finished by time $t$ and $\mathcal{F}(t)$ is the set of free tips whose POWs have been finished but have not yet been selected as parents at time $t$. Both $\mathcal{F}(t)$ and $\mathcal{V}'(t)$ are subsets of $\mathbb{N}$. Recall that a vertex which arrives at time $k$ is labeled by $k$. We now give each element inside the set $\mathcal{F}(t_i+\kappa_B)\cup\mathcal{V}'(t_i+\kappa_B)$ additional labels which will be used in the discussion through out step C: 
\begin{defn}\label{additionallabel}
Given $i$, a vertex $k$ in $\mathcal{F}(t_{i}+\kappa_B)\cup\mathcal{V}'(t_{i}+\kappa_B)$ is given an additional label as $(i,1,j)$ where: 1), $i$ represents the starting time of the bottleneck event that we are focusing on with $\{L(t_i)\leq b\}$ occurs; 2), we have a $1$ in this additional labeling system $(i,1,j)$ because $\mathcal{F}(t_i+\kappa_B)\cup\mathcal{V}'(t_i+\kappa_B)$ is the first set whose elements are given additional labels and we will also give additional labels to future arriving vertices in definition \ref{extendlabel}; 3), $j$ represents that $k$ is the $j-th$ smallest element in $\mathcal{F}(t_i+\kappa_B)\cup\mathcal{V}'(t_i+\kappa_B)$.  
\end{defn}
 As an example, given the bottleneck event that starts at time $t_i=47$, $\{L(t_{47})\leq b\}$ and $\mathcal{F}(t_{47}+\kappa_B)\cup\mathcal{V}'(t_{47}+\kappa_B)=\{2,7,50,51\}$ then vertices $2,7,50,51$ are given the additional labels $(47,1,1),(47,1,2),(47,1,3),$ $(47,1,4)$ respectively. 

For succinctness, we define the cardinality of the set $\mathcal{F}(t_i+\kappa_B)\cup\mathcal{V}'(t_i+\kappa_B)$ to be
\be 
c_i:=F(t_i+\kappa_B)+|\mathcal{V}'(t_i+\kappa_B)|,\label{defineci}
\ee 
which will be used in step C frequently.
\iffalse
Note that if given different $i$, which means we are considering bottleneck events starting at a different time, the additional label of the same vertex can be different. For instance, if instead of $i=47$ as provided in the last example, we are focusing on the bottleneck event that starts at time $i'=60$ and we are given that $\mathcal{F}(t_{60}+\kappa_B)\cup\mathcal{V}'(t_{60}+\kappa_B)=\{7,62,63,64\}$, then $7$ is also given an additional label as (60,1,1). However, since we are assuming that a specific value of $i$ is given, a vertex in $\mathcal{F}(t_i+\kappa_B)\cup\mathcal{V}'(t_i+\kappa_B)$ will only have one additional label in the form of $(i,1,j)$ in our construction of the bottleneck event.
\fi

Now that we have given the vertices in $\mathcal{F}(t_i+\kappa_B)\cup\mathcal{V}'(t_i+\kappa_B)$ the additional labels $(i,1,j)$, we are free to use either of the two types of label to referring to a specific vertex as long as we can recover the arriving time $k$ of the vertex given label in $(i,1,j)$. This is achieved by defining a function $\xi:\mathbb{N}\times \{1\}\times \mathbb{N}\rightarrow \mathbb{N}$ such that for $i\in\mathbb{N}$ and $j=1,2,...,c_i$,
\be
\xi(i,1,j):=\text{$j$-th smallest element in $\mathcal{F}(t_i+\kappa_B)\cup\mathcal{V}'(t_i+\kappa_B)$},\label{xi0case1}
\ee
where we will extend the domain of $\xi$ to $\mathbb{N}\times\mathbb{Z}^*\times \mathbb{N}$ in equation (\ref{renamexi}). With equation (\ref{xi0case1}), we can use the label $(i,1,j)$ to denote the vertex that has integer label $k=\xi(i,1,j)$. Using the previous example where $i=47$ and $\mathcal{F}(t_{47}+\kappa_B)\cup\mathcal{V}'(t_{47}+\kappa_B)=\{2,7,50,51\}$, we can use the function $\xi$ to find the integer label of vertex $(47,1,3)$ by evaluating $\xi(47,1,3)$ which equals 50.

A general demonstration for the relation between the additional label and the set $\mathcal{F}(t_i+\kappa_B)\cup\mathcal{V}'(t_i+\kappa_B)$ is given as follows:
\be 
\overbrace{ 
\underbrace{
 (i,1,1),(i,1,2),...,(i,1,2h_M+2\epsilon_{min})
 }_{\in\mathcal{F}_i^B}
 ,...,(i,1,F(t_i+\kappa_B))}^{\in\mathcal{F}(t_i+\kappa_B)},\nonumber\\
 \underbrace{(i,1,\mathcal{F}(t_i+\kappa_B)+1),...,(i,1,c_i)}_{\in\mathcal{V}'(t_i+\kappa_B)}\label{overbracexi}
\ee
where $(i,1,1),...,(i,1,2h_M+2\epsilon_{min})$ are the vertices in the set $\mathcal{F}_i^B$ whose elements remain unselected as parents until the time $t_i+\kappa_B$ if $B_i$ occurs by equation (\ref{stepBrule}). And note that any label in $\mathcal{V}'(t_i+\kappa_B)$ is less than any label in $\mathcal{F}(t_i+\kappa_B)$ because all the vertices involved has the same duration of POW being $h_M$ and any vertex in $\mathcal{V}'(t_i+\kappa_B)$ must have arrived later than any vertex in $\mathcal{F}(t_i+\kappa_B)$.

We consider the vertices with unfinished POW in $\mathcal{V}'(t_i+\kappa_B)$ because these are the future free tips once they finish their POW, which plays a critical role in later arguments including Lemma \ref{temp2} and step C.  Figure \ref{ExampleStepB} provides an example of this additional labeling system. 

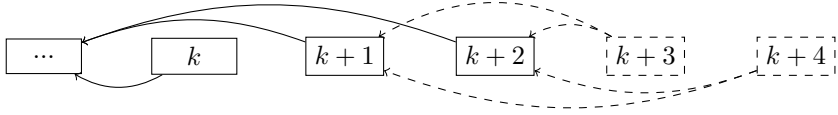
\begin{figure}
\centering
%\scalebox{0.65}{
\begin{tikzpicture}

    \node (0) [rectangle,draw,minimum height=0.46cm, minimum width=1cm] at (0,-2) {...};
    \node (1) [rectangle,draw] at (2,-2) {$\quad k \quad$};
    \node (2) [rectangle,draw] at (4,-2) {$k+1$};
    \node (3) [rectangle,draw] at (6,-2) {$k+2$};
    \node (4) [rectangle,draw,dashed] at (8,-2) {$k+3$};
    \node (5) [rectangle,draw,dashed] at (10,-2) {$k+4$};

    \graph{
        (1)->[bend left=30](0),
        (2)->[bend right=20](0),
        (3)->[bend right=20](0),
        (4)->[bend right=30,dashed](2),
        (4)->[bend right=30,dashed](3),
        (5)->[bend left=20,dashed](2),
        (5)->[bend left=20,dashed](3)
    
    };
\end{tikzpicture}
%}
\caption{Illustration of additional label with the form $(i,1,j)$ in Step B with a graph at time $t_i+\kappa_B$ with $i=79$. Here we use the same set up for solid and dashed rectangles as in Figure \ref{soD}. Solid vertex $k$ is a free tip at $t_{79}+\kappa_B$ and denoted as $(79,1,1)$. Solid vertices $k+1$ and $k+2$ have been selected as parents and they are now pending tips, hence they are not given additional labels. Dashed vertices $k+3,k+4$ are the vertices with unfinished POW at time $t_{79}+\kappa_B$ and are denoted as $(79,1,2)$ and $(79,1,3)$ respectively.}
\label{ExampleStepB}
\end{figure}

Next we establish equations (\ref{cibound}) and (\ref{remainasfreetip}) which will be used in discussion in step C including Lemma \ref{conditionC} and equation (\ref{kappac}).
Note that $F(t_i+\kappa_B)$ is bounded above by $b+\kappa_B M$ because of three reasons: First, $F(t_i)\leq L(t_i)\leq b$. Second, at most M POW can be finished at each time which determines the upper bound of increment of $F(\cdot)$. Finally, only fixed time $\kappa_B$ has passed after the time $t_i$. Furthermore, with the definition of $c_i$ in equation (\ref{defineci}), the number $c_i-F(t_i+\kappa_B)=|\mathcal{V}'(t_i+\kappa_B)|$ is equal to $h_M$ since only the vertices with labels in $\{t_i+\kappa_B-h_M,t_i+\kappa_B-h_M+1,...,t_i+\kappa_B-1\}$ have unfinished POW at time $t_i+\kappa_B$. Therefore, given $L(t_i)\leq b$ we have,
\be 
&F(t_i+\kappa_B)\leq b+M\kappa_B \quad\& &c_i\leq b+M\kappa_B+h_M \label{cibound}
\ee
Additionally, by equations (\ref{stepBrule}) and (\ref{xi0case1}), for $j\leq |\mathcal{F}_i^{B}|$ and any vertex $(i,0,j)$ which arrives at time $v=\xi(i,0,j)$,
\be
v\in \mathcal{F}(t_k) \quad \text{for all } k\in[t_i+\kappa_A,t_i+\kappa_B].\label{remainasfreetip}
\ee 

Next we highlight a central property to the bottleneck construction that arises in step B, for which we provide an illustration in Figure \ref{lemmaxi} which helps to demonstrate the idea. 

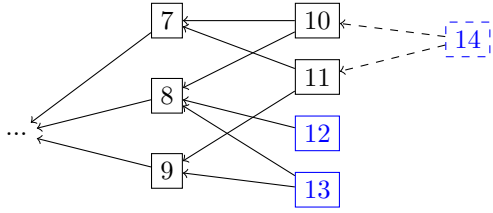
\begin{figure}
\centering
\begin{tikzpicture}
    \node (00) [rectangle] at (0,0) {...};
    
    \node (11) [rectangle,draw] at (2,1.5) {7};
    \node (12) [rectangle,draw] at (2,0.5) {8};
    \node (13) [rectangle,draw] at (2,-0.5) {9};
    
    \node (21) [rectangle,draw] at (4,1.5) {10};
    \node (22) [rectangle,draw] at (4,0.75) {11};
    \node (23) [rectangle,draw,color=blue] at (4,0) {12};
    \node (24) [rectangle,draw,color=blue] at (4,-0.75) {13};
    
    \node (31) [rectangle,draw,dashed,color=blue] at (6,1.25) {14};
    \graph{
    {(11),(12),(13)}->(00), (21)->{(11),(12)}, (22)->{(11),(13)}, (23)->{(12)},
    (24)->{(12),(13)}, (31)->[dashed]{(21),(22)}
    
    };
\end{tikzpicture}
\caption{Solid vertices and edges correspond to completed POWs, on the other hand, dashed vertices and edges represent unfinished POWs just like in Figure \ref{soD}. Here at $t_i+\kappa_B$ with $i=4$, the set of free tips $\mathcal{F}(t_i+\kappa_B)=\{12,13\}$, the set of pending tips $\mathcal{W}(t_i+\kappa_B)=\{10,11\}$ and the set of arrived vertices but with unfinished POW $\mathcal{V}'(t_i+\kappa_B)=\{14\}$. Vertices $12,13,14$ are colored in blue because they are given additional labels as $(4,1,1),(4,1,2)$ and $(4,1,3)$ respectively according to Definition \ref{additionallabel}. Black vertices are vertices that are accepted to the tangle but are not given an additional labels. Note that for any vertex with label less than 10, this vertex is or will be reachable by at least one element in $\{12,13,14\}$, i.e. any vertex in black is or will be reachable by at least one blue vertex. Generally speaking, for any vertex in $\mathcal{V}(t_i+\kappa_B)$ that are not in $\mathcal{F}(t_i+\kappa_B)$, this vertex is or will be reachable by at least one vertex in the set $\{(i,0,1),(i,0,2)),...,(i,0,c_i)\}$. The set of these blue vertices which has additional labels $\{(i,0,1),(i,0,2)),...,(i,0,c_i)\}$ plays like an input portal of the interchange structure that will be introduced in step C.}\label{lemmaxi}
\end{figure}

\begin{lem}\label{temp2}
Recall that $\mathcal{V}(t)$ is the set of vertices contained in the graph $\mathcal{G}(t)$. Suppose $L(t_i)\leq b$ and the events $A_i$ and $B_i$ corresponds to step A and B occur. For any graph $\mathcal{G}(t_n)$ with $t_n>t_i+\kappa_B+h_M$ and any vertices $v\in \mathcal{V}(t_i+\kappa_B)\setminus \mathcal{F}(t_i+\kappa_B)$, we have that $v$ is reachable from at least one element in $\{(i,1,1),(i,1,2),...,$ $(i,1,c_i)\}$.
\end{lem}

\begin{proof}
Let $v$ be any element in $\mathcal{V}(t_i+\kappa_B)\setminus \mathcal{F}(t_i+\kappa_B)$. Recall that $\mathcal{L}(t)$ is the set of tips at time $t$, which is the set of vertices in $\mathcal{V}(t)$ with in-degree 0 in the graph $ (\mathcal{V}(t),\mathcal{E}(t))$.

Suppose $v$ is a pending tip, which means $ v\in\mathcal{W}(t_i+\kappa_B)$ and $v$ has been selected as parent by at least one vertex whose POW is unfinished, then $v$ is reachable from a vertex $v'$ in $\mathcal{V}'(t_i+\kappa_B)$ and vertex $v'$ has addintional label as $(i,1,k)$ for some $k$ by Definition \ref{additionallabel}. 

Suppose $v$ is not a pending tip which means $v\notin \mathcal{W}(t_i+\kappa_B)$, then by the assumption that $v\in\mathcal{V}(t_i+\kappa_B)\setminus \mathcal{F}(t_i+\kappa_B)$ and the fact that $\mathcal{F}(t_i+\kappa_B)\cup \mathcal{W}(t_i+\kappa_B)=\mathcal{L}(t_i+\kappa_B)$, we have $v\in \mathcal{V}(t_i+\kappa_B)\setminus\mathcal{L}(t_i+\kappa_B)$. Hence $v$ is reachable from at least one element in $\mathcal{L}(t_i+\kappa_B)$ by the definition of tips. If $v$ is reachable from a $v'\in\mathcal{F}(t_i+\kappa_B)$ where $v'$ has additional label as $(i,1,k)$ for some $k$ by equation (\ref{overbracexi}), then we are done. If instead $v$ is reachable from an element in $\mathcal{W}(t_i-\kappa_B)$, then according to the second paragraph in this proof, $v$ is reachable from a pending tip which is further reachable from a vertex with additional label $(i,1,k)$ for some $k$.

Hence despite whether $v\in\mathcal{V}(t_i+\kappa_B)\setminus \mathcal{F}(t_i+\kappa_B)$ is a pending tip or not, $v$ is reachable from at least one element in $\{(i,1,1),(i,1,2),...,(i,1,c_i)\}$.
\end{proof}

\iffalse
\begin{figure}[H]
\centering
\includegraphics[scale=0.3]{4}
\caption{Result of Step B. Here we relabel the nodes. Circle vertices are free tips at $t_i+\kappa_B$, rectangle vertices are the vertices with unfinished POW at time $t_i+\kappa_B$ and they will become free tips when they finish their POW. }
\end{figure}
\fi
Lemma \ref{temp2} establishes that all vertices contained in $\mathcal{V}(t_i+\kappa_B)$ is connected via paths by vertices $\{(i,1,1),(i,1,2),...,(i,1,c_i)\}$. The vertex set $\{(i,1,1),(i,1,2),...,(i,1,c_i)\}$ can be viewed as an input port to the interchange structure that will be introduced in step C.

\textbf{Step C (Interchange Phase)}: Let $\kappa_C$ be a constant whose value will be defined later in equation (\ref{kappac}), we will focus on the interval $[t_i+\kappa_B,t_i+\kappa_C]$. Recall that $c_i:=F(t_i+\kappa_B)+|\mathcal{V}'(t_i+\kappa_B)|$. In order to construct the event $C_i$ corresponding to step C, we first extend Definition \ref{additionallabel} to the arrivals within the interval $[t_i+\kappa_B,t_i+\kappa_C]$ such that the first $c_i$ arrivals will be given additional labels as $(i,2,1),(i,2,2),...,(i,2,c_i)$ and the next $c_i$ arrivals will be given additional labels as $(i,3,1),(i,3,2),...,(i,3,c_i)$ and so on,
\begin{defn}\label{extendlabel}
 Given $i$, a vertex $v\in\mathbb{N}$ that arrivals within the interval $[t_i+\kappa_B,t_i+\kappa_C]$ is given an additional label as $(i,j,k)$ such that: 1), $i$ represents that the bottleneck event which we are focusing on starts at time $i$ with $\{L(t_i)\leq b\}$; 2), $j=2+\lfloor (v-t_i-\kappa_B)/c_i\rfloor$ and $k=\left[(v-t_i-\kappa_B) \mod c_i\right]+1$.  
\end{defn}
In the previous example where $i=47$, $\{L(t_{47})\leq b\}$ and $\mathcal{F}(t_{47}+\kappa_B)\cup\mathcal{V}'(t_{47}+\kappa_B)=\{2,7,50,51\}$ with $c_{47}=4$ 
 and $t_{47}+\kappa_B=52$, the vertices $2,7,50,51$ were given additional labels $(47,1,1), (47,1,2),(47,1,3),(47,1,4)$ respectively according to Definition \ref{additionallabel}. With Definition \ref{extendlabel}, the vertices with integer labels $52,53,54,55,56,57,...$ which arrive within the interval $[t_{47}+\kappa_B,t_{47}+\kappa_C]$ are given additional labels as $(47,2,1), (47,2,2), (47,2,3), (47,2,4), (47,3,1), (47,3,2), ...$ respectively. By Definition $\ref{extendlabel}$, given $i$, any vertex arrives within the interval $[t_i+\kappa_B,t_i+\kappa_C]$ has one additional label of the form $(i,j,k)$.  A general demonstration is given in equation (\ref{demonstrationxi}):

\be 
\overbrace{t_i+\kappa_B}^{(i,2,1)},\quad \overbrace{t_i+\kappa_B+1}^{(i,2,2)}, ...,\quad \overbrace{t_i+\kappa_B-1+c_i}^{(i,2,c_i)},\quad \overbrace{t_i+\kappa_B+c_i}^{(i,3,1)},\quad \overbrace{t_i+\kappa_B+c_i+1}^{(i,3,2)}, ...\label{demonstrationxi}
\ee

In order to establish that we can use either of the two types of label, we must ensure that we can find the arriving time $t$ of a vertex given its addition label $(i,j,k)$.  Base on equation (\ref{xi0case1}), we extend the domain of the function $\xi$ from $\mathbb{N}\times\{1\}\times\mathbb{N}$ to $\mathbb{N}\times \mathbb{Z}^*\times \mathbb{N}$ such that for $j\geq2$ and $k=1,...,c_i$,
\be
\xi(i,j,k):=t_i+\kappa_B-1+(j-2)c_i+k,\label{renamexi}
\ee
where $c_i=F(t_i+\kappa_B)+|\mathcal{V}'(t_i+\kappa_B)|$. Using equation (\ref{renamexi}) we can find that the vertex $(i,j,k)$ arrives at time $t=\xi(i,j,k)$ and hence this vertex also has integer label as $t=\xi(i,j,k)$. Using the previous example where $i=47$, $c_47=4$, $t_47+\kappa_B=52$ and a vertex is given an additional label $(47,2,3)$, using equation (\ref{renamexi}) we are able to recover the arriving time of the vertex by evaluating $\xi(47,2,3)$ which equals to $54$.

\iffalse
\begin{figure}[H]
\centering
\begin{tikzpicture}
    \node (01) [rectangle] at (-4.6,1) {$\xi_i(1,1)$};
    \node (11) at (-4.6,0.2) {};
    
    \node (02) [rectangle] at (-2.5,1) {$\xi_i(1,2)$};
    \node (12) at (-2.5,0.2) {};

    \node (03) [rectangle] at (-2.5,1) {$\xi_i(1,2)$};
    \node (13) at (-2.5,0.2) {};
    
    \node (10) [rectangle] at (0,0) {$t_i+\kappa_B,\quad t_i+\kappa_B+1,\quad ...,\quad t_i+\kappa_B-1+c_i,\quad t_i+\kappa_B+c_i, \quad...$};

    \graph {(01)->(11),(02)->(12)};
\end{tikzpicture}
\caption{}
\end{figure}
\fi

Within the interval $[t_i+\kappa_B,t_i+\kappa]$, we define an event $C_i$ that is an interchange structure which joins all rays together. Examples of the idea are provided in Figures \ref{stepC2'} and \ref{stepC2}. We define the event $C_i$ where each vertex arrives at $v\in[t_i+\kappa_B,t_i+\kappa]$ with additional label $(i,j,k)$ by Definition \ref{extendlabel} satisfies the following equations:
\begin{align} 
\Theta_v&=h_M \quad \&\quad \epsilon_v=\epsilon_{min}\label{stepcrule1}\\
X_v&=(i,j-1,\max(1,k-1))\label{Xselect}\\
Y_v&=(i,j-1,\min(k+1,c_i))\label{Yselect}
\end{align}
where $1_{\{\}}$ is the indicator function and $v=\xi(i,j,k)=t_i+\kappa_B-1+(j-2)c_i+k$ by equation (\ref{renamexi}). Equations (\ref{Xselect}) and (\ref{Yselect}) are saying that: \romannumeral 1), vertex $(i,j,k)$ chooses $(i,j-1,1)$ and $(i,j-1,2)$ as parents. \romannumeral 2), vertex $(i,j,c_i)$ chooses $(i,j-1,c_i-1)$ and $(i,j-1,c_i)$. \romannumeral 3), a vertex $(i,j,k)$ chooses $(i,j-1,k-1)$ and $(i,j-1,k+1)$ as parents for $1<k<c_i$. Figure \ref{stepC1} provides graphical explanation of Equation (\ref{Xselect}) and (\ref{Yselect}).

\begin{figure}
\centering
\begin{tikzpicture}
    \node(001)  at (-1,1.5) {...};
    \node(002) at (8,1.5) {...};
    \node (01) [rectangle,draw,color=blue] at (0,3) {\small $55$};
    \node (02) [rectangle,draw,color=blue] at (0,2) {\small $60$};
    \node (03) [rectangle,draw,color=blue] at (0,1) {\small $61$};
    \node (04) [rectangle,draw,color=blue] at (0,0) {\small $62$};

    \node (11) [rectangle,draw] at (1,3) {\small $63$};
    \node (12) [rectangle,draw] at (1,2) {\small $64$};
    \node (13) [rectangle,draw] at (1,1) {\small $65$};
    \node (14) [rectangle,draw] at (1,0) {\small $66$};

    \node (21) [rectangle,draw] at (2,3) {\small $67$};
    \node (22) [rectangle,draw] at (2,2) {\small $68$};
    \node (23) [rectangle,draw] at (2,1) {\small $69$};
    \node (24) [rectangle,draw] at (2,0) {\small $70$};

    \node (31) [rectangle,draw] at (3,3) {\small $71$};
    \node (32) [rectangle,draw] at (3,2) {\small $72$};
    \node (33) [rectangle,draw] at (3,1) {\small $73$};
    \node (34) [rectangle,draw] at (3,0) {\small $74$};   

    \node (41) [rectangle,draw] at (4,3) {\small $75$};
    \node (42) [rectangle,draw] at (4,2) {\small $76$};
    \node (43) [rectangle,draw] at (4,1) {\small $77$};
    \node (44) [rectangle,draw] at (4,0) {\small $78$};

    \node (51) [rectangle,draw] at (5,3) {\small $79$};
    \node (52) [rectangle,draw] at (5,2) {\small $80$};
    \node (53) [rectangle,draw] at (5,1) {\small $81$};
    \node (54) [rectangle,draw] at (5,0) {\small $82$};

    \node (61) [rectangle,draw] at (6,3) {\small $83$};
    \node (62) [rectangle,draw] at (6,2) {\small $84$};
    \node (63) [rectangle,draw] at (6,1) {\small $85$};
    \node (64) [rectangle,draw] at (6,0) {\small $86$};

    \node (71) [rectangle,draw,color=ao(english)] at (7,3) {\small $87$};
    \node (72) [rectangle,draw,color=ao(english)] at (7,2) {\small $88$};
    \node (73) [rectangle,draw,color=ao(english)] at (7,1) {\small $89$};
    \node (74) [rectangle,draw,color=ao(english)] at (7,0) {\small $90$}; 
    
    \foreach \x / \y in {1/0,2/1,3/2,4/3,5/4,6/5,7/6}{
            \graph{ (\x1) -> {(\y1),(\y2)}, (\x2) ->{(\y1),(\y3)},(\x3) ->{(\y2),(\y4)},(\x4)->{(\y3),(\y4)} };
        
    }
    %\graph {(11) -> {(01),(02)}, (12) ->{(01),(03)},(13) ->{(02),(04)},(14)->{(03),(04)},
    %(21) -> {(11),(12)}, (22) ->{(11),(13)},(23) ->{(12),(14)},(24)->{(13),(14)},
    %(31) -> {(21),(22)}, (32) ->{(21),(23)},(33) ->{(22),(24)},(34)->{(23),(24)},
    %(41) -> {(31),(32)}, (42) ->{(31),(33)},(43) ->{(32),(34)},(44)->{(33),(34)},
    %(51) -> {(41),(42)}, (52) ->{(41),(43)},(53) ->{(42),(44)},(54)->{(43),(44)}
    %};
\end{tikzpicture}
\caption{Example of construction in step C when $i=50$, $t_i+\kappa_B=63$, $\mathcal{F}(t_i+\kappa_B)\cup\mathcal{V}'(t_i+\kappa_B)=\{55,60,61,62\}$ and $c_i:=|\mathcal{F}(t_i+\kappa_B)\cup\mathcal{V}'(t_i+\kappa_B)|=4$. Note that vertices arrive by their order of the integer label instead of their position in the graph. For example, vertex 60 arrives at time $t=60$ and hence arrives before the vertex 63 which arrives at $t=63$. Every $c_i$ consecutive vertices that arrives within the time interval $[t_i+\kappa_B,t_i+\kappa_C]$ will contribute in one column of the interchange structure, for example vertices $75,76,77,78$ contribute to one column. Here the blue vertices are the vertices in the set $\mathcal{F}(t_i+\kappa_B)\cup\mathcal{V}'(t_i+\kappa_B)$ which contribute to the first column, the green vertices are the arrivals that contribute to the $(2c_i)$-th column and the black vertices are the arrivals that arrives within the blue and green arrivals. Any green vertex that contributes to the $(2c_i)$-th column will have a path to any blue vertex, for example the vertex $87$ has a path to each vertex $55,60,51$ and $62$. Then if each vertex that arrives after the green vertices has a path to one of the green vertices which further has a path to any blue vertex, using Lemma \ref{temp2} we will be able to show that any vertex that arrives after $t_i+\kappa_C$ has a path to any vertex in $\mathcal{V}(t_i)$, which is the critical property we want for the bottleneck events in order to prove one-endedness.} \label{stepC2'}
\end{figure}
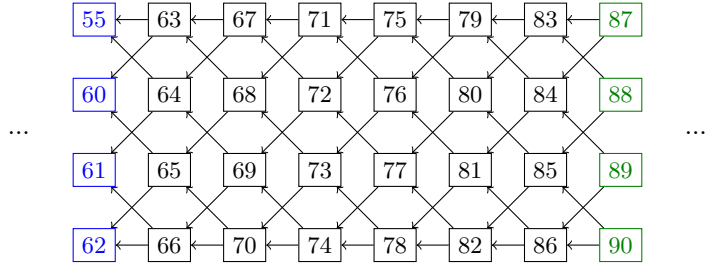

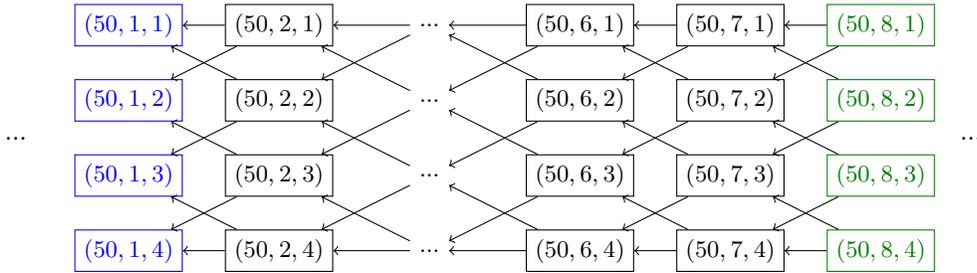
\begin{figure}
\centering
\begin{tikzpicture}
    \node(001)  at (-1.5,1.5) {...};
    \node(002) at (11.2,1.5) {...};
    \node (01) [rectangle,draw,color=blue] at (0,3) {\small $(50,1,1)$};
    \node (02) [rectangle,draw,color=blue] at (0,2) {\small $(50,1,2)$};
    \node (03) [rectangle,draw,color=blue] at (0,1) {\small $(50,1,3)$};
    \node (04) [rectangle,draw,color=blue] at (0,0) {\small $(50,1,4)$};

    \node (11) [rectangle,draw] at (2,3) {\small $(50,2,1)$};
    \node (12) [rectangle,draw] at (2,2) {\small $(50,2,2)$};
    \node (13) [rectangle,draw] at (2,1) {\small $(50,2,3)$};
    \node (14) [rectangle,draw] at (2,0) {\small $(50,2,4)$};

    \node (1)  at (4,3) {\small $...$};
    \node (2)  at (4,2) {\small $...$};
    \node (3)  at (4,1) {\small $...$};
    \node (4)  at (4,0) {\small $...$};

    \node (21) [rectangle,draw] at (6,3) {\small $(50,6,1)$};
    \node (22) [rectangle,draw] at (6,2) {\small $(50,6,2)$};
    \node (23) [rectangle,draw] at (6,1) {\small $(50,6,3)$};
    \node (24) [rectangle,draw] at (6,0) {\small $(50,6,4)$};

    \node (31) [rectangle,draw] at (8,3) {\small $(50,7,1)$};
    \node (32) [rectangle,draw] at (8,2) {\small $(50,7,2)$};
    \node (33) [rectangle,draw] at (8,1) {\small $(50,7,3)$};
    \node (34) [rectangle,draw] at (8,0) {\small $(50,7,4)$};

    \node (41) [rectangle,draw,color=ao(english)] at (10,3) {\small $(50,8,1)$};
    \node (42) [rectangle,draw,color=ao(english)] at (10,2) {\small $(50,8,2)$};
    \node (43) [rectangle,draw,color=ao(english)] at (10,1) {\small $(50,8,3)$};
    \node (44) [rectangle,draw,color=ao(english)] at (10,0) {\small $(50,8,4)$};   

    \foreach \x / \y in {4/3}{
            \graph{ (\x1) -> {(\y1),(\y2)}, (\x2) ->{(\y1),(\y3)},(\x3) ->{(\y2),(\y4)},(\x4)->{(\y3),(\y4)} };
        
    }

    \graph {(11) -> {(01),(02)}, (12) ->{(01),(03)},(13) ->{(02),(04)},(14)->{(03),(04)},
    (1) -> {(11),(12)}, (2) ->{(11),(13)},(3) ->{(12),(14)},(4)->{(13),(14)},
    (21)->{(1),(2)}, (22)->{(1),(3)}, (23)->{(2),(4)}, (24)->{(3),(4)},
    
    (31) -> {(21),(22)}, (32) ->{(21),(23)},(33) ->{(22),(24)},(34)->{(23),(24)}
    
    };
\end{tikzpicture}
\caption{Continuation of illustration in Figure \ref{stepC2'} and each vertex is labeled in the form of $(i,j,k)$ instead where $i=50$, $c_i=4$ and $t_i+\kappa_B=63$ in the example. For example, we can use $(50,1,1)$ and $55$ to denote the same vertex according to Definition \ref{additionallabel} and equation (\ref{xi0case1}). Another example is that we are using $(50,7,1)$ and $83$ to denote the same vertex according to Definition \ref{extendlabel} and equation (\ref{renamexi}). As shown in the graph, every $c_i$ consecutive vertices starting at time $t_i+\kappa_B$ contribute to one column of the interchange structure where a vertex with addition label $(i,j,k)$ contributes to the $j$-th column and $k$-th row. In general, the event $C_i$ has the property that any arrival $(i,j,k)$ with $j\geq 2$ and $k=1,2,...,c_i$ chooses $(i,j-1,\max(1,k-1))$ and 
$(i,j-1,\min(k+1,c_i))$ as parents where $1_{\{\}}$ is the indicator function. For example, here the vertex $(50,2,1)$ chooses $(50,1,1)$ and $(50,1,2)$ as parents while $(50,7,2)$ chooses $(50,6,1)$ and $(50,6,3)$ as parents.}\label{stepC2}
\end{figure}

\begin{figure}[H]
\centering
\begin{tikzpicture}
    \node at (-1.5,3) {...};
    \node at (11,3) {...};   
    \node (01) [rectangle,draw,color=blue] at (0,6) {\small $(i,1,1)$};
    \node (02) [rectangle,draw,color=blue] at (0,5) {\small $(i,1,2)$};
    \node (03) [rectangle,draw,color=blue] at (0,4) {\small $(i,1,3)$};
    \node (04) [rectangle,draw,color=blue] at (0,3) {\small $(i,1,4)$};
    \node (05)  at (0,2) {...};
    \node (06) [rectangle,draw,color=blue] at (0,1) {\small $(i,1,c_i-1)$};
    \node (07) [rectangle,draw,color=blue] at (0,0) {\small $(i,1,c_i)$};

    \node (11) [rectangle,draw] at (2,6) {\small $(i,2,1)$};
    \node (12) [rectangle,draw] at (2,5) {\small $(i,2,1)$};
    \node (13) [rectangle,draw] at (2,4) {\small $(i,2,1)$};
    \node (14)  at (2,3) {...};
    \node (15)  at (2,2) {...};
    \node (16) [rectangle,draw] at (2,1) {\small $(i,2,c_i-1)$};
    \node (17) [rectangle,draw] at (2,0) {\small $(i,2,c_i)$};  

    \node (21)  at (4,6) {...};
    \node (22)  at (4,5) {...};
    \node (23)  at (4,4) {...};
    \node (24)  at (4,3) {...};
    \node (25)  at (4,2) {...};
    \node (26)  at (4,1) {...};
    \node (27)  at (4,0) {...};  

    \node (31) [rectangle,draw] at (6.5,6) {\small $(i,2c_i-1,1)$};
    \node (32) [rectangle,draw] at (6.5,5) {\small $(i,2c_i-1,1)$};
    \node (33) [rectangle,draw] at (6.5,4) {\small $(i,2c_i-1,1)$};
    \node (34)  at (6.5,3) {...};
    \node (35)  at (6.5,2) {...};
    \node (36) [rectangle,draw] at (6.5,1) {\small $(i,2c_i-1,c_i-1)$};
    \node (37) [rectangle,draw] at (6.5,0) {\small $(i,2c_i-1,c_i)$};    
    
    \node (41) [rectangle,draw,color=ao(english)] at (9.5,6) {\small $(i,2c_i,1)$};
    \node (42) [rectangle,draw,color=ao(english)] at (9.5,5) {\small $(i,2c_i,1)$};
    \node (43) [rectangle,draw,color=ao(english)] at (9.5,4) {\small $(i,2c_i,1)$};
    \node (44)  at (9.5,3) {...};
    \node (45)  at (9.5,2) {...};
    \node (46) [rectangle,draw,color=ao(english)] at (9.5,1) {\small $(i,2c_i,c_i-1)$};
    \node (47) [rectangle,draw,color=ao(english)] at (9.5,0) {\small $(i,2c_i,c_i)$};    

     \foreach \x / \y in {1/0,2/1,3/2,4/3}{
            \graph{ (\x1) -> {(\y1),(\y2)}, 
            (\x2) ->{(\y1),(\y3)},
            (\x3) ->{(\y2),(\y4)},
            (\x4)->{(\y3),(\y5)}, 
            (\x5)->{(\y4),(\y6)},
            (\x6)->{(\y5),(\y7)},
            (\x7)->{(\y6),(\y7)}
            };
        
    }
    
    \graph {(11) -> {(01),(02)}, (12) ->{(01),(03)},(13) ->{(02),(04)},(16) ->{(05),(07)}, (17)->{(07),(06)},
    (21) -> {(11),(12)}, (22) ->{(11),(13)},(23) ->{(12),(14)},(26) ->{(15),(17)}, (27)->{(17),(16)}};
\end{tikzpicture}
\caption{Generalized illustration of the event $C_i$ based on Figures \ref{stepC2'} and \ref{stepC2}. Note that by equation (\ref{renamexi}), vertex $(i,j,k)$ arrives before vertex $(i,j,k')$ if $k'>k$ while vertex $(i,j,k)$ arrives before vertex $(i,j',k')$ if $j'>j$. In event $C_i$, the vertices are connected like meshes as shown in the graph and hence for any vertex $(i,j,k)$ with $j\geq 2c_i$ it has a path to any $(i,1,m)$ with $1\leq m\leq c_i$.}\label{stepC1}
\end{figure}

Recall that $\Sigma_i$ denotes the $\sigma$-algebra generated by all the random variables $\Theta_k,\epsilon_k,X_k$ and $Y_k$ for $k=1,2,...,i$. The conditional probability of event $C_i$ is given as follows:
\begin{lem}\label{conditionC}
 For any event $D_{i-1}\in \Sigma_{i-1}$ and $D_{i-1}\subseteq \{L(t_i)\leq b\}$, we have $P(B_i|D_{i-1}\cap A_i\cap B_i)>0$.
\end{lem}

\begin{proof}
As a reminder a vertex with additional label $(i,j,k)$ has the integer label $v=\xi(i,j,k)$ and this vertex arrives at time $v=\xi(i,j,k)$, which enables us to recover the arriving time and the integer label of a vertex from its additional label using the function $\xi()$. By equation (\ref{Xselect}), we first define the function $v_X(i,j,k):=\xi(i,j-1,\max(1,k-1))$ which gives the integer label for one of the desired parent of the vertex with additional label $(i,j,k)$. Similarly by equation (\ref{Yselect}), we define the function $v_Y(i,j,k):=\xi(i,j-1,\min(k+1,c_i))$ to represent the integer label for the other desired parent of the vertex with additional label $(i,j,k)$. Recall that $\mathcal{L}(t)$ is the set of tips at time $t$. In order to proof Lemma \ref{conditionC}, by equations (\ref{evolutiondistribution}), (\ref{stepcrule1}), (\ref{Xselect}) and (\ref{Yselect}), it suffices to show the following: given $\{L(t_i)\leq b\}$, $A_i$ and $B_i$ occur and any vertex $v$ with additional label $(i,j,k)$ arrives within the time interval $[t_i+\kappa_B,t_i+\kappa_C]$, then the vertices with integer labels $v_X(i,j,k):=\xi(i,j-1,\max(1,k-1))$ and $v_Y(i,j,k):=\xi(i,j-1,\min(k+1,c_i))$ are elements in $\mathcal{L}(t_v-\epsilon_{min})$.
 
 We first introduce some observations which will be used later in the proof: Observation 1 is that each vertex $(i,j,k)$ is selected as parent only by vertices with additional labels $(i,j+1,k-1)$, $(i,j+1,k)$, or $(i,j+1,k+1)$ according to equations (\ref{Xselect}) and (\ref{Yselect}) as well as equation (\ref{remainasfreetip}). Observation 2 is that a free tip at time $t$ must still be a tip at time $t+2$ because every arriving vertex has POW time $h_M$ and $h_M\geq 2$. 

 Here we consider five cases for vertex $v$ that arrives within $[t_i+\kappa_B,t_i+\kappa_C]$. Cases 1 and 2 cover the cases for the arrivals that has additional label $(i,j,k)$ such that $j=1$, i.e. the first $c_i$ vertices arrives starting from time $t_i+\kappa_B$. Cases 3,4 and 5 covers the rest of the vertices depending on their $k$ value in their additional label $(i,j,k)$.

 Case 1: If $v\in[i+\kappa_B,i+\kappa_B+|\mathcal{F}_i^{B}|-1)$ where $\mathcal{F}_i^{B}$ is the set that remains unselected as parents in step B, then by Definition \ref{extendlabel}, the additional label $(i,j,k)$ of the vertex $v$ satisfies that $j=2$ and $k\in\{1,2,...,|\mathcal{F}_i^{B}|\}$. Hence the two desired parents for vertex $v$ described in equation (\ref{Xselect}) and (\ref{Yselect}) are $(i,j-1,\max(1,k-1))=(i,1,\max(1,k-1))$ and $(i,j-1,\min(k+1,c_i))=(i,1,\min(k+1,c_i))$ where both of the vertices are free tips at time $t\in[t_i+\kappa_A,t_i+\kappa_B]$ by equations (\ref{remainasfreetip}). Then using Observation 1 and 2, we have the vertices with integer labels $v_X(i,j,k):=\xi(i,j-1,\max(1,k-1))$ and $v_Y(i,j,k):=\xi(i,j-1,\min(k+1,c_i))$ remain as elements of the set of tips and are counted in $\mathcal{L}(t)$ for $t\in[t_i+\kappa_A,t_v]$ and hence $v_X(i,j,k),v_Y(i,j,k)\in\mathcal{L}(t_v-\epsilon_{min})$.
 
 Case 2: For $v\in[i+\kappa_B+|\mathcal{F}_i^{B}|,i+\kappa_B+c_i-1]$, by Definition \ref{extendlabel}, vertex $v$ has additional label $(i,j,k)$ such that $j=2$ and $|\mathcal{F}_i^{B}|<k\leq c_i$.
 Hence the two desired parents for vertex $v$ described in equation (\ref{Xselect}) and (\ref{Yselect}) are $(i,j-1,\max(1,k-1))=(i,1,\max(1,k-1))$ and $(i,j-1,\min(k+1,c_i))=(i,1,\min(k+1,c_i))$. By equation (\ref{overbracexi}) and the fact that any unfinished POW at time $t_i+\kappa_B$ has duration $h_M$, we have all vertices $(i,1,F(t_i+\kappa_B)),
(i,1,F(t_i+\kappa_B)+1),...,(i,1,c_i)$ have completed their POWs and become free tips no later than time $t_i+\kappa_B+h_M$. Together with Observation 1 and 2 as well as the fact that $v-\epsilon_{min}\geq i+\kappa_B+h_M$ because $|\mathcal{F}_i^{B}|=2(h_M+\epsilon_{min})>h_M+\epsilon_{min}$,
we have that the vertices with integer labels $v_X(i,j,k):=\xi(i,1,\max(1,k-1))$ and $v_Y(i,j,k):=\xi(i,1,\min(k+1,c_i))$ remain as tips and are counted in $\mathcal{L}(t)$ for $t\in[t_i+\kappa_B+h_M,t_v]$  and hence $v_X(i,j,k),v_Y(i,j,k)\in\mathcal{L}(t_v-\epsilon_{min})$.

Cases 1 and 2 establish that all vertices $(i,2,1),(i,2,2),...,(i,2,c_i)$ have positive probability to choose their parents satisfying equation (\ref{Xselect}) and (\ref{Yselect}). Case 3, 4 and 5 establish the same results for any vertex $v\in [t_i+\kappa_B+c_i,t_i+\kappa_C]$ which has additional label as $(i,j,k)$ with $j\geq 3$ and $k\in\{1,2,...,c_i\}$. The idea is that there are at least $c_i-2$ vertices arrives between the arriving time of the vertex $(i,j,k)$ and of the two parents of $v_X(i,j,k),v_Y(i,j,k)$. Since $c_i>2(h_M+\epsilon_{min})$, by the time vertex $(i,j,k)$ arrives, $v_X(i,j,k)$ and $v_Y(i,j,k)$ have already finished POW for at least $\epsilon_{min}$ steps, then we have that $v_X(i,j,k),v_Y(i,j,k)\in\mathcal{L}(t_v-\epsilon_{min})$.

Case 3: For a vertex $v\in [t_i+\kappa_B+c_i,t_i+\kappa_C]$ such that the vertex has additional label $(i,j,k)$ with $j\geq 3$ and $k=1$, we have $(i,j-1,\max(1,k-1))=(i,j-1,1)$ and $(i,j-1,\min(k+1,c_i))=(i,j-1,2)$. The vertex $(i,j-1,1)$ arrives at time $\xi(i,j-1,1)=t_i+\kappa_B+(j-3)c_i$, finishes its POW at time $\xi(i,j-1,1)+h_M$ and counted as a free tip at time $\xi(i,j-1,1)+h_M+1$, i.e. $\xi(i,j-1,1)\in \mathcal{F}(\xi(i,j-1,1)+h_M+1)$. Using the fact that $c_i=F(t_i+\kappa_B)+|\mathcal{V}'(t_i+\kappa_B)|>2(h_M+\epsilon_{min})$, $h_M\geq 2$ and Observation 1, we have that the vertex $(i,j-1,1)$ is counted as free tip at time $\xi(i,j-1,1)+h_M+1\leq t_i+\kappa_B+(j-2)c_i-\epsilon_{min}=\xi(i,j,1)-\epsilon_{min}$, which means that $\xi(i,j-1,1)\in \mathcal{F}(\xi(i,j,k)-\epsilon_{min})$ and the vertex $(i,j-1,1)$ is counted as a free tip at least $\epsilon_{min}$ steps earlier than the arrival time of the vertex $(i,j,1)$. This together with Observation 1, we have the vertex $v_X(i,j,1)=\xi(i,j-1,1)$ is an element of $\mathcal{L}(t_v-\epsilon_{min})$.
Similarly, we have the vertex $v_Y(i,j,1)=\xi(i,j-1,2)$ is an element of $\mathcal{L}(t_v-\epsilon_{min})$.

Case 4: For a vertex $v\in [t_i+\kappa_B+c_i,t_i+\kappa_C]$ such that the vertex has additional label $(i,j,k)$ with $j\geq 3$ and $k\in \{2,3,...,c_i-1\}$, we have $(i,j-1,\max(1,k-1))=(i,j-1,k-1)$ and $(i,j-1,\min(k+1,c_i))=(i,j-1,k+1)$. The vertex $(i,j-1,k+1)$ arrives at time $\xi(i,j-1,k+1)=t_i+\kappa_B+(j-3)c_i+k$, finishes its POW at time $\xi(i,j-1,k+1)+h_M$ and counted as a free tip at time $\xi(i,j-1,k+1)+h_M+1$, i.e. $\xi(i,j-1,k+1)\in\mathcal{F}(\xi(i,j-1,k+1)+h_M+1)$. Using the fact that $c_i>2(h_M+\epsilon_{min})$, $h_M\geq 2$ and Observation 1, we have $\xi(i,j-1,k+1)+h_M+1\leq t_i+\kappa_B+(j-2)c_i-\epsilon_{min}=\xi(i,j,k)-\epsilon_{min}$, which means that $\xi(i,j-1,k+1)\in\mathcal{F}(\xi(i,j,k)-\epsilon_{min})$ and the vertex $(i,j-1,k+1)$ is counted as a free tip at least $\epsilon_{min}$ steps earlier than the arrival time of the vertex $(i,j,k)$. This together with Observation 1, we have the vertex $v_Y(i,j,k)=\xi(i,j-1,k+1)$ is an element of $\mathcal{L}(t_v-\epsilon_{min})$. Similarly, the vertex $(i,j-1,k-1)$ that arrives at time $\xi(i,j-1,k-1)=t_i+\kappa_B+(j-2)c_i+k-2$ finishes its POW at time $\xi(i,j-1,k-1)+h_M$ and counted as a free tip at time $\xi(i,j-1,k-1)+h_M+1$, i.e. $\xi(i,j-1,k-1)\in\mathcal{F}(\xi(i,j-1,k-1)+h_M+1)$. Using the fact that $c_i>2(h_M+\epsilon_{min})$ and $h_M\geq 2$ we have $\xi(i,j-1,k-1)+h_M+1\leq t_i+\kappa_B+(j-2)c_i-\epsilon_{min}=\xi(i,j,k)-\epsilon_{min}$, which means that $\xi(i,j-1,k-1)\in\mathcal{F}(\xi(i,j,k)-\epsilon_{min})$ and the vertex $(i,j-1,k-1)$ is already a free tip at least $\epsilon_{min}$ steps earlier than the arrival of the vertex $(i,j,k)$. Here although the vertex $(i,j-1,k-1)$ is also selected as parent by the vertex $(i,j,k-2)$, by Observations 1 and 2 we still have that the vertex $v_X(i,j,k)=\xi(i,j-1,k-1)$ is an element of $\mathcal{L}(t_v-\epsilon_{min})$.

Case 5: For a vertex $v\in [t_i+\kappa_B+c_i,t_i+\kappa_C]$ such that the vertex has additional label $(i,j,k)$ with $j\geq 3$ and $k=c_i$, we have $(i,j-1,\max(1,k-1))=(i,j-1,c_i-1)$ and $(i,j-1,\min(k+1,c_i))=(i,j-1,c_i)$. Here the situation is similar to Case 4 where we we need to use both Observation 1 and 2, this is because both vertex $(i,j-1,c_i-1)$ and vertex $(i,j-1,c_i)$
has been selected as parents by vertices that arrive before vertex $(i,j,c_i)$ arrives. Hence using the same process in case 4, we have vertices $v_X(i,j,k)=\xi(i,j-1,c_i-1)$ and 
$v_Y(i,j,k)=\xi(i,j-1,c_i)$ are elements in $\mathcal{L}(t_v-\epsilon_{min})$.

Cases 1 and 2 covers any vertex that arrives within the interval $[t_i+\kappa_B, t_i+\kappa_B+c_i-1]$. For any vertex $v$ that arrives within $[t_i+\kappa_B+c_i,t_i+\kappa_C]$, by Definition \ref{extendlabel}, the additional label of vertex $v$ must have $j\geq 3$ and $k=1,2,...,c_i$. Hence depending on the value of $k$, vertex $v$ arrives within $[t_i+\kappa_B+c_i,t_i+\kappa_C]$ must be included in either cases 3, 4 or 5. Therefore, cases 1 through 5 cover all the situations for any vertex  $v\in [t_i+\kappa_B+c_i,t_i+\kappa_C]$, we are done.
\end{proof}

In order to establish Lemma \ref{temp3}, we now identify the value of $\kappa_C$ as shown in equation (\ref{kappac}) with the following ideas: 1), as demonstrated in Figures \ref{stepC2'} and \ref{stepC2}, the desired number of columns in the interchange structure is at least $2c_i$. Hence, with equation (\ref{cibound}), we select $\kappa_C>\kappa_B+2(b+M\kappa_B+h_M)^2$ such that $(\kappa_C-\kappa_B)/c_i>2c_i$ in order to achieve this goal; 2), after the vertices $(i,2c_i,1),(i,2c_i,2),...,(i,2c_i,c_i)$ arrives we prolong step C with time $h_M+\epsilon_{max}+1$ so that these vertices are included in the DAG $\mathcal{G}(t)$ for $t\geq t_i+\kappa_C-\epsilon_{max}$. Recall that a vertex $k$ select parents within the set of tips $\mathcal{L}(t_k-\epsilon_k)$ where $\epsilon_k$ is a random variable bounded above by $\epsilon_{max}$. Then by prolonging step C, we make sure that every vertices that arrives later than $t_i+\kappa_C$ will select their parents depending on the DAG where vertices $(i,2c_i,1),(i,2c_i,2),...,(i,2c_i,c_i)$ are included.
\be
\kappa_C:=\kappa_B+2(b+M\kappa_B+h_M)^2+h_M+\epsilon_{max}+1>\kappa_B+2c_i^2+h_M+\epsilon_{max}+1.\label{kappac}
\ee

With $\kappa_C$ defined, we have completed the construction of the bottleneck event $\mathcal{B}_i:=\{L(t_i)\leq b\}\cup A_i\cup B_i \cup C_i$. We now introduce Lemma \ref{temp3} which establishes the result that any vertex which arrives after $t_i+\kappa_C$ has a path to one of the element in $\{(i,2c_i,1),(i,2c_i,2),...,(i,2c_i,c_i)\}$ and hence further to any element in $\{(i,1,1),(i,1,2),$ $...,(i,1,c_i)\}$. Lemma \ref{temp3} will be used in proving one-endedness.
\begin{lem} \label{temp3}
If given $i$ such that the event $\{L(t_i\leq b)\}$ and the events $A_i,B_i$ and $C_i$ described by equations (\ref{stepA1}), (\ref{stepA}), (\ref{stepBrule1}), (\ref{stepBrule}), (\ref{stepcrule1}), (\ref{Xselect}) and (\ref{Yselect}) occur, then: 1), any vertex $v>t_i+\kappa_C$ has directed path to an element in $\{(i,2c_i,1),(i,2c_i,2),...,(i,2c_i,c_i)\}$ in the graph $\cup_t^{\infty}\mathcal{G}(t)$; 2), any vertex $v>t_i+\kappa_C$ has directed path to any element in $\{(i,1,1),(i,1,2),...,(i,1,c_i)\}$ in the graph $\cup_t^{\infty}\mathcal{G}(t)$.
\end{lem}

\begin{proof}
Here we first provide some observations based on equations (\ref{stepcrule1}), (\ref{Xselect}) and (\ref{Yselect}) and Figure \ref{stepC1}. Observation 1: For any vertex which arrives within the interval $[t_i+\kappa_B,t_i+\kappa_C]$ and has additional label $(i,j,k)$ such that $j> 2c_i$, this vertex has a directed path to a vertex with label $(i,2c_i,k')$ for some $k'$.

Observation 2: By equation (\ref{renamexi}), vertices $(i,2c_i,1),(i,2c_i,2),...,(i,2c_i,c_i)$ arrives no later than $t_i+\kappa_B+2c_i^2$ and counted as free tip no later than $t_i+\kappa_B+2c_i^2+h_M+1$. Therefore by equation (\ref{kappac}), $(i,2c_i,1),(i,2c_i,2),...,(i,2c_i,c_i)\in \mathcal{V}(t)$ for any $t>t_i+\kappa_C-\epsilon_{max}$ where $\mathcal{V}(t)$ is the set of vertices that have completed POW at time $t$. Hence by equations (\ref{stepcrule1}), (\ref{Xselect}) and (\ref{Yselect}) and Figure \ref{stepC1}, we have that for any $t>t_i+\kappa_C-\epsilon_{max}$, the set of leaves only contains vertices that arrives after the vertex $(i,2c_i,1)$ which means $\mathcal{L}(t)\subseteq \{n|n\geq \xi(i,2c_i,1)\}$.

We now prove the first statement in Lemma \ref{temp3} by considering two cases.

Case 1: For any vertex $v>t_i+\kappa_C$, suppose one of the parents of the vertex $v$ is a vertex with integer label $v'$ such that $v'\leq t_i+\kappa_C$. Then by Definition \ref{extendlabel}, $v'$ has additional label $(i,j',k')$ for some $j'$ and $k'$. Also, since vertex $v'$ is a parent of $v$, we have that $v'\in\mathcal{L}(v-\epsilon_v)$ where $v-\epsilon_v>t_i+\kappa_C-\epsilon_{max}$. By Observation 2, the additional label $(i,j',k')$ of vertex $v'$ satisfies that $j'\geq 2c_i$. Hence by Observation 1, $v$ has a directed path to a vertex with label $(i,2c_i,k'')$ for some $k''$.

Case 2: Suppose otherwise that both of the parents of the vertex $v$ have integer label bigger than $t_i+\kappa_C$. Then let vertex $v'$ be one of the parents of the vertex $v$ and note that $v'<v$ by the fact that vertex $v'$ must have arrived before vertex $v$ does so that $v'$ can be selected as a parent of vertex $v$. Since vertex $v$ has a directed path to vertex $v'$, in order to prove the first statement under the consideration of case 2, we can instead prove vertex $v'$ has a directed path to an element in $\{(i,2c_i,1),(i,2c_i,2),$\\$...,(i,2c_i,c_i)\}$. We can repeat this process until we find a parent with integer label $v''$ such that $v''\leq t_i+\kappa_C$ and use the proof for case 1. This iteration must end because the integer label for a vertex is always bigger than the integer label of any parent of this vertex. Hence the first statement of Lemma \ref{temp3} is true.

We now prove the second statement. By the first statement, it suffices to show that for any value of $k$ and any vertex with additional label $(i,2c_i,k')$, vertex $(i,2c_i,k')$ has a directed path to the vertex $(i,1,k)$. We will use the intuition we introduced in Figure \ref{stepC2'},\ref{stepC2} and \ref{stepC1}. By equations (\ref{Xselect}) and (\ref{Yselect}), $(i,2c_i,k)$ has a directed path to vertex $(i,2c_i-k'+1,1)$. Since $k,k'\leq c_i$, $2c_i-k'+1>k$vertex and $(i,2c_i-k'+1,1)$ has a directed path to vertex $(i,k,1)$ which further has a directed path to vertex $(i,1,k)$. Hence the second statement of Lemma \ref{temp3} is true.
\end{proof}

By Lemma \ref{temp3}, in the case where $\mathcal{B}_i$ happens, all rays are joined to each other within this interchange structure in the sense of equivalence of rays. For an event $\mathcal{B}_i$ we define the following sets of vertices using variables defined in steps A,B and C:
\bee 
\mathcal{V}_i^-:&=&\mathcal{V}(t_i+\kappa_B), \text{ which is the set of vertices included in the DAG at } t_i+\kappa_B\\
\mathcal{V}_i^0:&=&\{ v| \text{vertex } v \text{ that has an additional label in } \{(i,1,1),(i,1,2),...,(i,1,c_i)\}\} \\
\mathcal{V}_i^+:&=& \{v|v> t_i+\kappa_C, v\in\mathbb{N}\}
\eee 

We now state Lemma \ref{temp4} which is central to prove one-endedness.
\begin{lem} \label{temp4}
Given $\mathcal{B}_i$ occurs, for any $v^-\in \mathcal{V}_i^-$ and $v^{+}\in \mathcal{V}_i^+$, there exist a directed path from $v^+$ to $v^-$ in the graph $\cup_t^{\infty}\mathcal{G}(t)$.
\end{lem}

\begin{proof}
Recall that $\mathcal{F}(t)$ is the set of free tips at time $t$. Suppose vertex $v^-$ in $\mathcal{V}_i^-$. If $v^-\in\mathcal{F}(t_i+\kappa_B)$, then by Definition \ref{additionallabel}, there exists $v^0\in \mathcal{V}_i^0$ such that $v^0=v^-$. Else $v^-\in\mathcal{V}(t_i+\kappa_B)\setminus \mathcal{F}(t_i+\kappa_B)$, then by Lemma \ref{temp2} there exist at least one vertex $v^0\in\mathcal{V}_i^0$ such that $v^-$ is reachable from $v^0$. Secondly, from Lemma \ref{temp3}, all vertices in $\mathcal{V}_i^0$ are reachable from all the vertices in $\mathcal{V}_i^+$. Hence Lemma \ref{temp4} follows.
\end{proof}

\iffalse
\begin{figure}[H]
\centering
\begin{tikzpicture}
    \node (01) [rectangle,draw,align=left] at (-5,0) {\small All the nodes accepted \\to tangle before the\\ time $t_i+\kappa_B$};
    \node (02) [rectangle,draw,align=left] at (-0.5,0) {\small $(i,1,1),(i,1,2),...,(i,1,c_i)$\\ defined in \\construction of $\mathcal{B}_i$};
    \node (03) [rectangle,draw,align=left] at (3.8,0) {\small All future arrivals after \\time $t_i+\kappa_C$.\\ };
    \node (11) at (-5,-1.5) {$\mathcal{V}_i^-$};
    \node (12) at (-1,-1.5) {$\mathcal{V}_i^0$};
    \node (13) at (3,-1.5) {$\mathcal{V}_i^+$};

    \graph{(03)->(02)};
    \draw[->, dashed] (02) -- (01) ;
\end{tikzpicture}
\caption{ Dashed arrow implies that any element in $\mathcal{V}_i^-$ has a directed path towards it from at least one element in $\mathcal{V}_i^0$. Complete arrow implies that all vertices in $\mathcal{V}_i^+$ have directed paths to all vertices in $\mathcal{V}_i^0$. }
\label{p8}
\end{figure}
\fi

Recall that $\mathcal{S}_*$ denote the space of connected DAGs rooted at vertex $0$ with all vertices having finite degrees and $d_*(\mathcal{G}_1,\mathcal{G}_2):=(r+1)^{-1}$ where $r$ is the biggest integer such that the two r-balls rooted at vertex $0$ in $\mathcal{G}_1$ and $\mathcal{G}_2$ are identical. Using Lemma \ref{temp4}, we provide Lemma \ref{cauchy1} which will be used to proof that the limit of the sequence $(\mathcal{G}(t_n))_{n\in\mathbb{N}}$ exists in the metric space $(\mathcal{S}_*,d_*)$. As a reminder, $\mathcal{L}(t)$ is the set of tips which includes all the vertices that have in-degree $0$.
\begin{lem}\label{cauchy1}
Suppose the bottleneck event $\mathcal{B}_i$ happens. For any $t\geq t_i+\kappa_C-\epsilon_{max}-h_M$, we have that $\mathcal{L}(t_i)\cap\mathcal{L}(t)=\emptyset$.  
\end{lem}
\begin{proof}
Suppose for contradiction we have a $t\geq t_i+\kappa_C-\epsilon_{max}-h_M$ and a vertex $v$ such that $v\in\mathcal{L}(t_i)\cap\mathcal{L}(t)$. Since a vertex will never become a tip again once it cease to be a tip, together with the fact that $\kappa_C-\kappa_B>\epsilon_{max}+h_M$ by equation (\ref{kappac}), we get $v\in \mathcal{L}(t_i+\kappa_B)$. Since the set of tips consist of free tips and pending tips, i.e. $\mathcal{L}(t_i+\kappa_B)=\mathcal{F}(t_i+\kappa_B)\cup\mathcal{W}(t_i+\kappa_B)$, we will consider two cases depending on whether $v$ is a free tip at $t_i+\kappa_B$ or a pending tip to derive contradiction.

Case 1: if $v\in\mathcal{F}(t_i+\kappa_B)$, then by equations (\ref{Xselect}) and (\ref{Yselect}) we have that  $v$ is chosen as parent at some time within the interval $[t_i+\kappa_B,t_i+\kappa_B+c_i-1]$ and become a pending tip at the next step. Since maximum duration for a POW is $h_M$, vertex $v$ must cease to be a pending tip at some time within 
$[t_i+\kappa_B+h_M+1,t_i+\kappa_B+c_i+h_M]$. By equation (\ref{kappac}), $t_i+\kappa_B+c_i+h_M<t_i+\kappa_C-\epsilon_{max}-h_M$, and hence $v$ has cease to be a tip before time $t_i+\kappa_C-\epsilon_{max}-h_M$ and it can never be a tip again. Therefore $v\notin\mathcal{L}(t_i+\kappa_C-\epsilon_{max}-h_M)$ and $v\notin\mathcal{L}(t)$, which leads to contradiction. 

Case 2: if $v\in\mathcal{W}(t_i+\kappa_B)$, then vertex $v$ has been selected as parent and it will get attached to no later than $t_i+\kappa_B+h_M$. Using the same process in case 1, we have $v\notin \mathcal{L}(t)$, which leads to contradiction.
\end{proof}

Now that we have constructed the bottleneck events, we will establish that the bottleneck events happen infinitely often. Recall that $\Sigma_i$ denotes the $\sigma$-algebra generated by all the random variables $\Theta_k,\epsilon_k,X_k$ and $Y_k$ for $k=1,2,...,i$. Based on lemma \ref{L<b} and the construction of step A, B and C, we will show that $\mathcal{B}_i$ happens infinitely often using the following idea: At each step there are at most $M$ vertices whose POWs are completed there for the number of tips can increase by at most M at each step. At the same time, each of these $M$ vertices, whose POWs have just been finished, can cause at most 2 vertices to cease to be a tip. Therefore, given $\{L(t_i)\leq b\}$ occurs where $L(t_i)$ is the number of tips at time $t_i$, the number of tips is bounded above by $b+ 2M\max(|\epsilon_{max},\kappa_C-\epsilon_{min}|)<b+2M\kappa_C$ within the interval $(t_i-\epsilon_{max},t_i+\kappa_C-\epsilon_{min}]$ and hence the probability of selecting a specific tip as a parent of a vertices that arrives within $[t_i,t_i+\kappa]$ is also bounded below by a positive constant $1/(b+2\kappa_C M)$. Thus for the bottleneck event $\mathcal{B}_i:=\{L(t_i\leq b)\}\cup A_i\cup B_i \cup C_i$, which is defined through equations (\ref{stepA1}), (\ref{stepA}), (\ref{stepBrule1}), (\ref{stepBrule}), (\ref{stepcrule1}), (\ref{Xselect}) and (\ref{Yselect}), is a union of sets that can be written in the form $\{L(t_i)\leq b\}\cap(\cap_{j=i}^{\kappa_C}\{\Theta_j=\theta_j,\epsilon_j=e_j,X_j=x_j,Y_j=y_j\})$. Hence given an event $D_{i-1}$ such that $D_{i-1}\in \Sigma_{i-1}$ and $D_{i-1}\subseteq \{L(t_i)\leq b\}$, by equation (\ref{evolutiondistribution}), we have that $P(\mathcal{B}_i|D_{i-1})>0$ implies
\be 
P(\mathcal{B}_i|D_{i-1})&\geq& \left(\frac{\min\{p_{\Theta,j}|j=1,2,...,M\}\times\min\{p_{\epsilon,j}|j\in I_\epsilon\}}{(b+2\kappa_C M)}\right)^{\kappa_C}=:\rho(\kappa_C)\label{pkappa}
\ee
where $p_{\Theta,j}=P(\Theta_k=h_j)$ for any $k$, $I_{\epsilon}$ denotes the set of possible values of $\epsilon_k$ and $p_{\epsilon,j}=P(\epsilon_k=j)$ for $j\in I_\epsilon$. Since $P(\mathcal{B}_i|D_{i-1})>0$ by Lemmas \ref{conditionA}, \ref{conditionB} and \ref{conditionC}, $P(\mathcal{B}_i|D_{i-1})\geq \rho(\kappa_C)$. This fact provides sufficient condition to show that the bottleneck events happen infinitely often as described in Lemma \ref{Binfinitelyoften}. $\mathcal{B}_i$ is a special event that forces the branches in the DAG stick together again, the fact that $\mathcal{B}_i$ happens infinitely often is central to prove theorem \ref{mainresult}.
\begin{lem}\label{Binfinitelyoften}
The sequence of events $\{\mathcal{B}_i\}_{i=1}^\infty$  happens infinitely often.
\end{lem}

\begin{proof}
Let $\mathcal{A}_i=\{L(t_i)\leq b\}$ as mentioned in lemma \ref{L<b} with $b$ a constant in $(10 h_M-6h_1+3M\epsilon_{max} +2,\infty)$. The idea is that given $\mathcal{A}_i$ occurs the probability of event $\mathcal{B}_i$ happening is bounded below by $\rho(\kappa_C)$ as defined in equation (\ref{pkappa}), and together with the Lemma \ref{L<b} saying that $\{\mathcal{A}_i\}_{i=1}^\infty$ happen infinitely often we can prove $\{\mathcal{B}_i\}_{i=1}^\infty$ happen infinitely often.\\
Recall that $\Sigma_i$ denotes the $\sigma$-algebra generated by all the random variables $\Theta_k,\epsilon_k,X_k$ and $Y_k$ for $k=1,2,...,i$. Note that $L(t_i)$ is $\Sigma_{i-1}$ measurable by equations (\ref{evov}) and (\ref{evoe}). 

By equation (\ref{pkappa}), for any event $D_{i-1}$ such that $D_{i-1}\in \Sigma_{i-1}$ and $D_{i-1}\subseteq \{L(t_i)\leq b\}$,
\be 
P(\mathcal{B}_i|D_{i-1})\geq \rho(\kappa_C)\label{lowerboundBi}
\ee

This means that as long as we know $\mathcal{A}_i:=\{L(t_i\leq b)\}$ happens, no matter what happened before time $t_i$, the conditional probability of $\mathcal{B}_i$ will be always bounded below by $\rho(\kappa_C)$. Note that $\mathcal{B}_i$ is not $\mathcal{F}_{i}$ measurable since the bottleneck event covers the interval $[t_i,t_i+\kappa_C]$.

Since $\mathcal{A}_i$ happens infinitely often by Lemma \ref{L<b}, we can define:
\bee 
S_1&:=&\min \{i|\mathcal{A}_{i} \text{ happens } \}\label{s1}\\
S_j&:=&\min \{i|\mathcal{A}_{i} \text{ happens and } i-s_{j-1}>2\kappa_C\} \quad\quad j=2,3,4,...\label{sj}
\eee 
i.e., there exist a sequence of time $t_{S_1},t_{S_2},...$ such that these $A_{S_j}$ occur and there is sufficient time between each of them. Figure \ref{si} provide a demonstration.

\begin{figure}[H]
\begin{tikzpicture}
    % draw horizontal line   
    \draw[ultra thick, ->] (0,0) -- (10,0);
    
    % draw vertical lines
    \foreach \x in {1,2,3,4,5,6,7,8,9}
    \draw (\x cm,3pt) -- (\x cm,-3pt);

    \node at (1,0.5) {$t_1$};
    \node at (2,0.5) {$t_2$};
    \node at (3,0.5) {$t_3$};
    \node at (4,0.5) {$t_4$};
    \node at (5,0.5) {$t_5$};
    \node at (6,0.5) {$t_6$};
    \node at (7,0.5) {$t_7$};
    \node at (8,0.5) {$t_8$};
    \node at (9,0.5) {$t_9$};

    \node (01)at (1,-0.5) {$\mathcal{A}_1$};
    \node at (2,-0.5) {$\mathcal{A}_2^c$};
    \node at (3,-0.5) {$\mathcal{A}_3$};
    \node (02) at (4,-0.5) {$\mathcal{A}_4$};
    \node at (5,-0.5) {$\mathcal{A}_5^c$};
    \node at (6,-0.5) {$\mathcal{A}_6^c$};
    \node at (7,-0.5) {$\mathcal{A}_7^c$};
    \node at (8,-0.5) {$\mathcal{A}_8^c$};
    \node (03) at (9,-0.5) {$\mathcal{A}_9$};

    \node (11) at (1,-1.5) {$S_1=1$};
    \node (12) at (4,-1.5) {$S_2=4$};
    \node (13) at (9,-1.5) {$S_3=9$};

    \node (21) at (1,-2.5) {$Z_1=?$};
    \node (22) at (4,-2.5) {$Z_2=?$};
    \node (23) at (9,-2.5) {$Z_3=?$};

    \graph{(01)->(11),(02)->(12),(03)->(13),
    (11)->(21),(12)->(22),(13)->(23)};

\end{tikzpicture}
\caption{An example of the definition of $s_j$ with $\kappa_C=1$. The situation about whether $\mathcal{A}_i$ happens is shown, for example, $\mathcal{A}_1$ is displayed showing that $\mathcal{A}_1$ is true while $\mathcal{A}_2^c$ is displayed showing that $\mathcal{A}_2$ is false. Then by equations (\ref{s1}) and (\ref{sj}) we can calculate the values of $S_j$ as shown in the figure. The definition of random variable $Z_j$ is that $\{Z_j=1\}\iff \{ \mathcal{B}_{S_j} \text{ occurs}\}$. Our goal is to focus on some of the times when $\mathcal{A}_i$ occurs while the times stay sufficiently far away from each other. The distance requirement makes sure that the events $(\mathcal{B}_{S_j})_{i=1}^\infty$ do not have intersection interval and hence equation (\ref{lowerboundBi}) can be applied.}\label{si}
\end{figure}
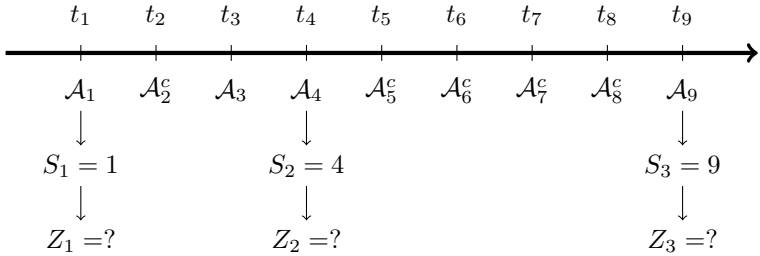

Define the random variable $Z_j$ as: $\{Z_j=1\}\iff \{ \mathcal{B}_{S_j} \text{ occurs}\}$. Note that for any $j$, when we talk about $\mathcal{B}_{S_j}$ we already have that $A_{S_j}$ occurs. By equation (\ref{lowerboundBi}),
\bee 
&&\P(Z_j=1|Z_{j-1},Z_{j-2},...Z_1,S_1,S_2,...,S_j)\geq\rho(\kappa_C)
\eee
Denote $\vec{Z}_j:=(Z_1,Z_2,...,Z_j)$, $\vec{S_j}:=(S_1,S_2,...,S_j)$ and $\vec{s_j}:=(s_1,s_2,...,s_j)$. We have
\bee
&&\P(Z_j=1|Z_{j-1},Z_{j-2},...Z_1)=\P(Z_j|\vec{Z}_{j-1})\\
&&=\sum_{\vec{s}_{j-1}}\P(Z_j=1|\vec{Z}_{j-1},\vec{S}_{j-1}=\vec{s}_{j-1})\P(\vec{S}_{j-1}=\vec{s}_{j-1}|\vec{Z}_{j-1})\geq\rho(\kappa_C)
\eee 

Hence 
\bee 
&&P(\cap_{j=m}^{\infty} \{Z_j\neq 1\}) \leq \prod_{j=m}^{\infty} (1-\rho(\kappa_C))=0\\
&&\implies P(\cup_{m=1}^{\infty}\cap_{j=m}^{\infty} \{Z_j\neq 1\}) =0
\eee 

Hence $\{Z_j=1\}$ happens infinitely often and hence $\mathcal{B}_i$ happens infinitely often.

\end{proof}

Recall that $\mathcal{S}_*$ denote the space of connected DAGs rooted at vertex $0$ with all vertices having finite degrees and $d_*(\mathcal{G}_1,\mathcal{G}_2):=(r+1)^{-1}$ where $r$ is the biggest integer such that the two r-balls rooted at vertex $0$ in $\mathcal{G}_1$ and $\mathcal{G}_2$ are identical. Using Lemmas \ref{cauchy1} and \ref{Binfinitelyoften}, we now establish the part about existence of limit stated in Theorem  \ref{mainresult}.
\begin{lem}\label{temp5}
In the metric space $(\mathcal{S}_*,d_*)$, the sequence $(\mathcal{G}(t_i))_{i\in\mathbb{N}}$ is almost surely Cauchy and hence the limit exist and equals to $\cup_{i=1}^{\infty} \mathcal{G}(t_i)$.
\end{lem}

\begin{proof}
First, given any time $t$, any vertex $v$ in the graph $\mathcal{G}(t)$ has finite degree and hence $\mathcal{G}(t)\in (\mathcal{S}_*,d_*)$ because : 1), it can has almost 2 parents which means it's out-degree is at most 2; 2),  it can has at most $h_M$ in-degree because once vertex $v$ is selected as a parent for the first time by a vertex $v'$, there are only $h_M$ steps left before the POW corresponding to $v'$ is finished and then vertex $v$ will be attached by $v'$ which means vertex $v$ can no longer be selected as parent. Furthermore, because of this, $\cup_{i=1}^{\infty}\mathcal{G}(t_i)$ has finite degree and $\cup_{i=1}^{\infty}\mathcal{G}(t_i)\in (\mathcal{S}_*,d_*)$.

Next we show that the limit exist. As established in \cite{Da07}, the metric space $(\mathcal{S}_*,d_*)$ is separable and complete, i.e. a Polish space. To show the existence of the limit, it suffices to show that the sequence $(\mathcal{G}(t_i))_{i\in\mathbb{N}}$ is Cauchy.

Recall $\mathcal{V}(t)$ is the set of vertices at time $t$. We define the function $d(t,v_1,v_2)$ such that it gives the shortest path distance from vertex $v_1$ to $v_2$ in the graph $\mathcal{G}(t)$. Notice that for any path from vertex $v$ to $v''$ in $\cup_{i=1}^{\infty}\mathcal{G}(t_i)$, it has the form $(v_0=v,v_1,v_2,...,v_{k-1},v_{k}=v'')$ where $v_{j}$ is a parent of $v_{j-1}$ for all $j=1,2,...,k$ and $v_{j}$ is a vertex in the graph $\mathcal{G}$ by the time $v_{j-1}$ arrives to select parents. Therefore, any path from vertex $v\in\mathbb{N}$ to vertex $v'$ is already included in the graph $\mathcal{G}(t_{v})$ and so is the shortest path from $v$ to $v'$. Hence we can use the notation $d(v_1,v_2)$ instead.

Define the function $D(t_1,t_2)$ which gives the biggest integer such that the two r-balls rooted at vertex $0$ in $\mathcal{G}(t_1)$ and $\mathcal{G}(t_2)$ are identical. Recall $\mathcal{L}(t)$ is the set of tips at time $t$.

By Lemma \ref{Binfinitelyoften}, almost surely there exists a sequence $s_1<s_2<...$ such that $\mathcal{B}_{s_i}$ occurs and $s_i-s_{i-1}>2\kappa_C$. The idea is that each bottleneck event acts like a protective shell so that anything added after a bottleneck event can only change the part of the graph that is outside the shell instead of any r-ball contained within the shell. Take $r_0:=\min\{d(v',0)|v'\in\mathcal{L}(s_1)\}$. For any vertex $v$ that is added to the DAG $\mathcal{G}()$ at time greater than $s_1+\kappa_C$, it arrives at time $v\geq s_1+\kappa_C-h_M$ and it selected parents within $\mathcal{L}(v-\epsilon_v)$ where $v-\epsilon_v\geq s_1+\kappa_C-\epsilon_{max}-h_M$. Hence by Lemma \ref{cauchy1}, $d(v,0)>r_0$ which means anything added to the DAG after $s_1+\kappa_C$ will not affect the $r_0$-ball. Hence $D(t,t')\geq r_0$ for any $t,t'>s_1+\kappa_C$. By the same reason, we have $r_1:=\min\{d(v',0)|v'\in\mathcal{L}(s_2)\}$ is greater or equal to $r_0+1$ where $s_2-s_1>2\kappa_C$. Using induction on $i$ for $i=1,2,...$, we conclude that $D(t_1,t_2)\geq r_{i-1}=r_0+i-1$ for $t,t'>s_{i}+\kappa_C$ and $r_{i}:=\min\{d(v',0)|v'\in\mathcal{L}(s_{i+1})\}$ is greater or equal to $r_{i-1}+1=r_0+i$.  Hence for any $k>r_0\geq 0$ we can find a $T=k-r_0+1$ such that for any $t,t'>s_T+\kappa_C$, we have $d_*(\mathcal{G}(t),\mathcal{G}(t'))=1/(1+D(t_1,t_2))\leq 1/(1+k)$. Hence the sequence $(\mathcal{G}(t_i))_{i\in\mathbb{N}}$ is almost surely Cauchy. As established in \cite{Da07}, the metric space $(\mathcal{S}_*,d_*)$ is separable and complete, i.e. a Polish space. Hence the limit of the sequence exist $(\mathcal{G}(t_i))_{i\in\mathbb{N}}$. Since for any $\delta>0$, there exist a T such that for any $t,t'>T$, $d_*(\mathcal{G}(t),\mathcal{G}(t'))\leq \delta$, which means $d_*(\mathcal{G}(t),\cup_{i=1}^{\infty} \mathcal{G}(t_i))\leq \delta$, we get that the limit equals to $\cup_{i=1}^{\infty} \mathcal{G}(t_i)$.
\end{proof}

We are now ready to establish one-endedness and complete the proof of Theorem \ref{mainresult}. 

\begin{proof}[Proof of Theorem \ref{mainresult}] By Lemma \ref{temp5}, all that left to show in Theorem \ref{mainresult} is one-endedness. To prove this result, by definition \ref{d31}, \ref{d32} and \ref{d33}, it suffices to show for any two rays $r_1,r_2$, there exist a third ray $r_3$ that intersect both of them infinitely often. We will construct the third ray $r_3$ using the bottleneck events. 

Notice that since all the edges must goes from the vertex with bigger label to the vertex with the smaller label, then a ray must be an increasing sequence or else it will not be a infinite sequence. Let $v_0:=0$ which is the root of the whole DAG and $v_1$ be the second vertex in $r_2$ where any vertex has a path to root which is vertex $0$ . Since $\mathcal{B}_i$ happens infinitely often, we can find $i_1$ such that $B_{i_1}$ happens after $v_1$ finishes its POW. By Lemma \ref{temp4}, $v_1$ will be linked by all the vertices arrives starting at $t_{i_1}+\kappa_C$. Therefore there must be a vertex $v_2$ in $r_2$ such that it has a path to $v_1$, and we will use this path to construct $r_3=(v_0,...,v_1,....,v_2,....)$. Since $\mathcal{B}_i$ happens infinitely often there will be no problem repeating this process to find $i_k$ such that $i_{k-1}-i_k>2\kappa_C$ and get the third ray $r_3=(v_0,...,v_1,...,v_2,...,v_3,...,v_4,...)$ where $v_k\in r_2$ if $k$ is odd and $v_k\in r_1$ if $k$ is even. The vertices between $v_j$ to $v_{j+1}$ in $r_3$ are the path from $v_{j+1}$ to $v_j$ which exists by Lemma \ref{temp4}.
\end{proof}

Theorem \ref{mainresult} establishes the crucial result that the limiting DAG $\mathcal{G}(\infty)$ which models the dynamic of IOTA distributed ledger  is one-ended, such result is important and essential to consensus of distributed ledger. If a vertex $v$ is reachable from a vertex $v'$, then we say that vertex $v$ is verified by $v'$ because one will need to redo the POW corresponding to $v'$ to alter the information in $v$. When talking about digital ledger, we are often interested in the number of vertices that are verified by all but finitely many vertices. We refer to these vertices as confirmed vertices because the the information contained in these nodes are secured by all but finitely many future vertices. 
\begin{defn}
A vertex $v$ is confirmed in the infinite DAG is confirmed if all but finitely many later vertices have paths to $v$.
\end{defn}
By Lemmas \ref{temp4} and \ref{Binfinitelyoften}, we have Theorem \ref{confirmed} which establish the result that each vertices in the DAG are secured by almost all future vertices and the computational power used to finish POWs are efficiently used to secure the information stored in the distributed ledger.
\begin{thm}\label{confirmed}
    For the model assumed in Section \ref{model}. With probability 1, all vertices are confirmed in the DAG $\mathcal{G}(\infty):=\lim_{i\rightarrow\infty}(\mathcal{G}(t_i))$.
\end{thm}

Since different kinds of parent selection algorithms are used in practice, a natural question arises as to whether one-endedness holds for more general situation. The idea used to prove Theorem \ref{mainresult} with modification of proof to lemma \ref{L<b} can be adapted accordingly. In the model defined in section \ref{model}, if each vertices select $k>1$ parents instead, then the infinite DAG is still one-ended since the supermartingale we construct in Lemma \ref{L<b} will still tend to go down when the number of free tips is sufficiently large. By the same logic, it can be proved that one-ended property also holds if each vertex selects a random number of parents where the number is drawn from an i.i.d process. For example, each arrival has probability $p$ of choosing 3 parents and $1-p$ choosing 1 parent.
\begin{prop}\label{generization}
    Assume a similar model as in Section \ref{model} except that each vertex $i$ selects $k$ parents with replacement from $\mathcal{L}(t_i-\epsilon_{i})$. Then Theorems \ref{mainresult} and \ref{confirmed} still hold.
\end{prop}

There are many important topics related to distributed ledgers, some of which arise from the variety of models used to define the dynamics of a distributed ledger. For example, generalization of the main result to the model assuming continuous distribution for vertex arriving time and duration of POW is also interesting.

%%%%%%%%%%%%Reference list%%%%%%%%%%%%%%
%
% References should be in the following form (or the BibTeX file
% apt.bst should be used):
%
% For a journal:
% Surname, Initial (year). Title of paper. {\em Journal title}
% {\bf Vol,} page--range.
%
% For a book:
% Surname, Initial (year). {\em Book title}. Publisher, Address.
%
% Note the following example of a reference list.

\end{document}